\documentclass{svproc}

\usepackage{url}

\usepackage{amsmath,amssymb,amsfonts,mathtools}

\DeclareMathOperator{\Cov}{Cov}
\DeclareMathOperator{\Ent}{Ent}
\DeclareMathOperator{\Expectation}{\mathbb E} 
\DeclareMathOperator{\Grad}{grad}
\DeclareMathOperator{\Hessian}{Hess}
\DeclareMathOperator{\Sign}{sign}

\newcommand{\Bspace}[1]{B_{#1}}
\newcommand{\Ccs}[1]{C_0\left(#1\right)}
\newcommand{\Cexp}[1]{C_0^{(\cosh-1)}\left(#1\right)}
\newcommand{\Cinfcs}[1]{C_0^\infty\left(#1\right)}
\newcommand{\Cinfp}[1]{C_{\mathrm{p}}^\infty\left(#1\right)}

\newcommand{\Lexp}[1]{L^{(\cosh-1)}\left(#1\right)}
\newcommand{\LlogL}[1]{L^{(\cosh-1)_*}\left(#1\right)}

\newcommand{\WCexp}[1]{C_0^{1,(\cosh-1)}\left( #1 \right)}
\newcommand{\Wexp}[1]{W^{1,(\cosh-1)}\left(#1\right)}
\newcommand{\WlogL}[1]{W^{1,(\cosh-1)_*}\left(#1\right)}
\newcommand{\absoluteval}[1]{\left|#1\right|}
\newcommand{\avalof}[1]{\absoluteval{#1}}
\newcommand{\avalsof}[1]{\absoluteval{#1}^2}
\newcommand{\bnabla}{{\vec\nabla}}

\newcommand{\covat}[3]{\Cov_{#1}\left(#2,#3\right)}

\newcommand{\derivby}[1]{\frac{d}{d{#1}}}

\newcommand{\entropyof}[1]{\Ent\left(#1\right)}

\newcommand{\euler}{\mathrm{e}}
\newcommand{\expectat}[2]{{\Expectation}_{#1}\left[#2\right]}

\newcommand{\expof}[1]{\exp\left(#1\right)}
\newcommand{\gaussdensity}{M}
\newcommand{\gaussint}[2]{\int{#1} \gaussdensity(#2) \ d#2 \ }
\newcommand{\hessianof}[1]{\Hessian #1}
\newcommand{\logof}[1]{\log\left(#1\right)}
\newcommand{\maxexp}[1]{{\mathcal E}\left(#1\right)}
\newcommand{\naturals}{\mathbb N}
\newcommand{\normat}[2]{\left\Vert#2\right\Vert_{#1}}
\newcommand{\normof}[1]{\left\Vert#1\right\Vert}
\newcommand{\partialdd}[1]{\frac{\partial^2}{\partial{#1^2}}}
\newcommand{\partiald}[1]{\frac{\partial}{\partial{#1}}}
\newcommand{\probabilities}{\mathcal P}
\newcommand{\reals}{\mathbb R}
\newcommand{\scalarat}[3]{\left\langle#2,#3\right\rangle_{#1}}
\newcommand{\scalarof}[2]{\left\langle#1,#2\right\rangle}
\newcommand{\sdomainat}[1]{\mathcal S_{#1}}
\newcommand{\sdomain}[1]{\mathcal S_{#1}}
\newcommand{\setof}[2]{\left\{#1 \middle| #2 \right\}}
\newcommand{\signof}[1]{\Sign\left(#1\right)}
\newcommand{\transpose}{*}

\begin{document}
\mainmatter              
\title{Information Geometry of the Gaussian Space}
\titlerunning{Information Geometry of the Gaussian Space}  
%
\author{Giovanni Pistone}
\authorrunning{G. Pistone} 
%
\tocauthor{Giovanni Pistone}
\institute{de Castro Statistics, Collegio Carlo Alberto, Piazza Vincenzo Arbarello 8, 10122 Torino, Italy,\\
\email{giovanni.pistone@carloalberto.org},\\ WWW home page:
\texttt{http://www.giannidiorestino.it}}

\maketitle              

\begin{abstract}
  We discuss the Pistone-Sempi exponential manifold on the finite-dimensional Gaussian space. We consider the role of the entropy, the continuity of translations, Poincar\'e-type inequalities, the generalized differentiability of probability densities of the Gaussian space.
  
\end{abstract}

\section{Introduction}

The Information Geometry (IG) set-up based on exponential Orlicz spaces \cite{MR1370295}, as further developed in \cite{MR1704564,MR1628177,MR2396032,MR3130268,MR3126029,MR3474821}, has reached a satisfying consistency, but has a basic defect. In fact, it is unable to deal with the structure of the measure space on which probability densities are defined. When the basic space is $\reals^n$ one would like to discuss for example transformation models as sub-manifold of the exponential manifold, which is impossible without some theory about the effect of transformation of the state space on the relevant Orlicz spaces. Another example of interest are evolution equations for densities, such as the Fokker-Planck equation, which are difficult to discuss in this set-up without considering Gaussian Orlicz-Sobolev spaces. See an example of such type of applications in \cite{MR1693600,brigo|pistone:2017CIG,arXiv:1603.04348v2}.

In \cite{MR3365132} the idea of an exponential manifold in a Gaussian space has been introduced and the idea is applied to the study of the spatially homogeneous Boltzmann equation. In the second part of that paper, it is suggested that the Gaussian space allows to consider Orlicz-Sobolev spaces with Gaussian weight of \cite[Ch. II]{MR724434} as a set-up for exponential manifolds.

In Sec. \ref{sec:gaussian-space} we discuss some properties of the Gauss-Orlicz spaces. Most results are quite standard, but are developed in some detail because to the best of our knowledge the case of interest is not treated in standard treatises. Notable examples and Poincar\'e-type inequalities are considered in Sec. \ref{sec:ineq-motiv}.  

The properties of the exponential manifold in the Gaussian case that are related with the smoothness of translation and the existence of mollifiers are presented in Sec. \ref{sec:expon-manif-gauss}. A short part of this section is based on the conference paper \cite{pistone:2017-GSI2017}. Gaussian Orlicz-Sobolev space are presented in Sec. \ref{sec:orlicz-sobolev}. Only basic notions on Sobolev's spaces are used here, mainly using the presentation by Haim Brezis \cite[Ch. 8--9]{MR2759829}.

Part of the results presented here were announced in an invited talk at the conference IGAIA IV Information Geometry and its Applications IV, June 12-17, 2016, Liblice, Czech Republic.

\section{Orlicz spaces with Gaussian Weight}
\label{sec:gaussian-space}

All along this paper, the sample space is the real Borel space $(\reals^n,\mathcal B)$ and $M$ denotes the standard $n$-dimensional Gaussian density ($M$ because of J.C. Maxwell!),

\begin{equation*}
  M(x)=(2\pi)^{-n/2}\exp\left(-\frac{1}{2}\avalof x ^2\right), \qquad x \in \reals^{n}\ .
\end{equation*}

\subsection{Generalities}

First, we review basic facts about Orlicz spaces. Our reference on Orlicz space is J. Musielak monograph \cite[Ch. II]{MR724434}.

On the probability space $(\reals^n,\mathcal B,M)$, called here the \emph{Gaussian space}, the couple of Young functions $(\cosh - 1)$ and its conjugate $(\cosh-1)_*$ are associated with the Orlicz space $\Lexp M$ and $\LlogL M$, respectively.

The space $\Lexp M$ is called \emph{exponential space} and is the vector space of all functions such that $\gaussint {(\cosh-1)(\alpha f(x))} x < \infty$ for some $\alpha > 0$. This is the same as saying that the moment generating function $t \mapsto \gaussint {\euler^{tf(x)}} x$ is finite on a open interval containing 0.

If $x,y \ge 0$, we have $(\cosh-1)'(x)= \sinh(x)$, $(\cosh-1)'_*(y)=\sinh^{-1}(y) = \logof{y + \sqrt{1+y^2}}$, $(\cosh-1)_*(y) = \int_0^y \sinh^{-1}(t) \ dt$. The Fenchel-Young inequality is
\begin{equation*}
  xy \le (\cosh-1)(x) + (\cosh-1)_*(y) = \int_0^x \sinh(s) \ ds + \int_0^y \sinh^{-1}(t) \ dt
\end{equation*}
and
\begin{align*}
  (\cosh-1)(x) &= x \sinh(x) - (\cosh-1)_*(\sinh(x)) \ ; \\
(\cosh-1)_*(y) &= y \sinh^{-1}(y) - (\cosh-1)(\sinh^{-1}(y)) \\
&= y \logof{y+(1+y^2)^{1/2}} - (1+y^2)^{-1/2} \ .
\end{align*}

The conjugate Young function $(\cosh-1)_*$ is associated with the \emph{mixture space} $\LlogL M$. In this case, we have the inequality
\begin{equation}\label{eq:delta2}
(\cosh-1)_*(ay) \le C(a) (\cosh-1)_*(y), \quad C(a) = \max(\avalof a,a^2) \ .  
\end{equation}
In fact
\begin{multline*}
  (\cosh-1)_*(ay) = \int_0^{ay} \frac{ay - t}{\sqrt{1+t^2}} \ dt = \\ a^2 \int_0^y \frac{y-s}{\sqrt{1+a^2s^2}} \ ds = a \int_0^y \frac{y-s}{\sqrt{\frac1{a^2}+s^2}} \ ds\ .
\end{multline*}
The inequality \eqref{eq:delta2} follows easily by considering the two cases $a > 1$ and $a < 1$. As a consequence, $g \in \LlogL M$ if, and only if, $\gaussint {(\cosh-1)_*(g(y))} y < \infty$.

In the theory of Orlicz spaces, the existence of a bound of the type \eqref{eq:delta2} is called $\Delta_2$-property, and it is quite relevant. In our case, it implies that the mixture space $\LlogL M$ is the dual space of its conjugate, the exponential space $\Lexp M$. Moreover, a separating sub-vector space e.g., $\Cinfcs{\reals^n}$ is norm-dense.

In the definition of the associated spaces, the couple $(\cosh - 1)$ and $(\cosh-1)_*$ is equivalent to the couple defined for $x,y > 0$ by $\Phi(x)=\euler^x-1-x$ and $\Psi(y) = (1+y)\logof{1+y} - y$. In fact, for $t > 0$ we have $\logof{1+t} \le \logof{y+\sqrt{1+t^2}}$ and 
\begin{equation*}
  \logof{t+\sqrt{1+t^2}} \le \logof{t+\sqrt{1+2t+t^2}} = \logof{1+2t} \ , 
\end{equation*}
so that we derive by integration the inequality
\begin{equation*}
  \Psi(y) \le (\cosh-1)_*(y) \le \frac12 \Psi(2y) \ .
\end{equation*}
In turn, conjugation gives
\begin{equation*}
  \frac12 \Phi(x) \le (\cosh-1)(x) \le \Phi(x) \ .
\end{equation*}

The \emph{exponential space} $\Lexp M$ and the \emph{mixture space} $\LlogL M$ are the spaces of real functions on $\reals^n$ respectively defined using the conjugate Young functions $\cosh-1$ and $(\cosh -1)_*$. The exponential space and the mixture space are given norms by defining the closed unit balls of $\Lexp M$ and $\LlogL M$, respectively, by
\begin{equation*}
  \setof{f}{\int (\cosh-1)(f(x)) \ M(x)dx \le 1}, \quad \setof{g}{\int (\cosh-1)_*(g(x)) \ M(x)dx \le 1} \ .
\end{equation*}
Such a norm is called Luxemburg norm.

The Fenchel-Young inequality
\begin{equation*}
  xy \le (\cosh-1)(x) + (\cosh-1)_*(y)
\end{equation*}
implies that $(f,g) \mapsto \expectat M {fg}$ is a separating duality, precisely
\begin{equation*}
  \avalof{\gaussint {f(x)g(x)} x } \le 2 \normat {\Lexp M} f \normat {\LlogL M} g \ .
\end{equation*}

A random variable $g$ has norm $\normat{\LlogL M} g$ bounded by $\rho$ if, and only if, $\normat{\LlogL M}{g/\rho} \le 1$, that is $\expectat M {(\cosh-1)_*(g/\rho)} \le 1$, which in turn implies $\expectat M {(\cosh-1)_*(\alpha g)} = \expectat M {(\cosh-1)_*(\alpha\rho(g/\rho))}\le \rho\alpha$ for all $\alpha\ge 0$. This is not true for the exponential space $\Lexp M$.

It is possible to define a dual norm, called Orlicz norm, on the exponential space, as follows. We have $\normat{(\LlogL M)^*} f \le 1$ if, and only if $\avalof{\gaussint {f(x)g(x)} x} \le 1$ for all $g$ such that $\gaussint {(\cosh-1)_*(g(x))} x \le 1$. With this norm, we have
\begin{equation}
  \label{eq:orlicznorm}
  \avalof{ \gaussint {f(x)g(x)} x } \le \normat {(\LlogL M)^*} {f} \normat {\LlogL M} {g} 
\end{equation}

The Orlicz norm and the Luxemburg norm are equivalent, precisely,
\begin{equation*}
  \normat{\Lexp M} f \le \normat{\LlogL M^*} f \le 2\normat{\Lexp M} f \ .
\end{equation*}

\subsection{Entropy}

The use of the exponential space is justified by the fact that for every 1-dimensional exponential family $I \ni \theta \mapsto p(\theta) \propto \euler^{\theta V}$, $I$ neighborhood of 0, the sufficient statistics $V$ belongs to the exponential space. The statistical interest of the mixture space resides in its relation with entropy. 

If $f$ is a positive density of the Gaussian space, $\gaussint {f(x)} x = 1$, we define its entropy to be $\entropyof f = - \gaussint {f(x)\log f(x)} x$. As $x\log x \ge x-1$, the integral is well defined. It holds
\begin{equation}\label{eq:logplus}
 - \gaussint {f(x) \log^+ f(x)} x \le \entropyof f \le \euler^{-1} - \gaussint {f(x) \log^+ f(x)} x \ ,
\end{equation}
where $\log^+$ is the positive part of $\log$.

\begin{proposition} \label{prop:LlogL-entropy}
A positive density $f$ of the Gaussian space has finite entropy if, and only if, $f$ belongs to the mixture space $\LlogL M$.   
\end{proposition}

\begin{proof}
  We use Eq. \eqref{eq:logplus} in order to show the equivalence. For $x \ge 1$ it holds
  \begin{equation*}
    2x \le x + \sqrt{1+x^2} \le (1+ \sqrt 2) x \ .
  \end{equation*}

It follows
  \begin{equation*}
 \log 2 + \log x \le \logof{x + \sqrt{1+x^2}} = \sinh^{-1}(x) \le \logof{1+ \sqrt 2} + \log x \ ,
  \end{equation*}
and, taking the integral $\int_1^y$ with $y \ge 1$, we get
\begin{multline*}
  \log 2 (y -1) + y \log y - y + 1 \le \\
(\cosh-1)_*(y) - (\cosh-1)_*(1) \le \\
\logof{1+ \sqrt 2}(y-1) + y \log y - y + 1 \ ,
\end{multline*}
then, substituting $y>1$ with $\max(1,f(x))$, $f(x) > 0$,
\begin{multline*}
  (\log 2 - 1)(f(x)-1)^+ + f(x) \log^+ f(x) \le \\
(\cosh-1)_*(\max(1,f(x))) - (\cosh-1)_*(1) \le \\
(\logof{1+ \sqrt 2}-1)(f(x)-1)^+ + f(x) \log^+ f(x) \ .
\end{multline*}
By taking the Gaussian integral, we have
\begin{multline*}
  (\log 2 - 1) \gaussint{(f(x)-1)^+} x + \gaussint{f(x) \log^+ f(x)} x \le \\
\gaussint{(\cosh-1)_*(\max(1,f(x)))}x - (\cosh-1)_*(1) \le \\
(\logof{1+ \sqrt 2}-1)\gaussint{(f(x)-1)^+}x + \gaussint{f(x) \log^+ f(x)} x \ ,
\end{multline*}
which in turn implies the statement because $f \in L^1(M)$ and
\begin{multline*}
  \gaussint {(\cosh-1)_*(f(x))} x  + (\cosh-1)_*(1) = \\ \gaussint {(\cosh-1)_*(\max(1,f(x)))} x + \gaussint {(\cosh-1)_*(\min(1,f(x)))} x \ .
\end{multline*}
\qed
\end{proof}

Of course, this proof does not depend on the Gaussian assumption.

\subsection{Orlicz and Lebesgue spaces}

We discuss now the relations between the exponential space, the mixture space, and the Lebesgue spaces. This provides a first list of classes of functions that belong to the exponential space or to the mixture space. The first item in the proposition holds for a general base probability measure, while the other is proved in the Gaussian case.
 
\begin{proposition}\label{prop:O-inclusions} Let $1 < a < \infty$.
\begin{enumerate} 
\item       \begin{equation*}
        L^\infty(M) \hookrightarrow \Lexp M \hookrightarrow L^a(M) \hookrightarrow \LlogL M \hookrightarrow L^1(M) \ .
      \end{equation*}
    \item If $\Omega_R = \setof{x \in \reals^n}{\absoluteval x < R}$, the restriction operator is defined and continuous in the cases 
      \begin{equation*}
        \Lexp M \rightarrow L^a(\Omega_R), \quad  \LlogL M \rightarrow L^1(\Omega_R)
      \end{equation*}
    \end{enumerate}
  \end{proposition}
    
\begin{proof}\begin{enumerate}
  \item See \cite[Ch. II]{MR724434}.
\item For all integers $n \ge 1$,
  \begin{multline*}
    1 \ge \int (\cosh-1)\left(\frac{f(x)}{\normat{\Lexp M}{f}}\right) M(x) \ dx \ge \\ \int_{\Omega_R} \frac1{(2n)!} \left(\frac{f(x)}{\normat{\Lexp M}{f}}\right)^{2n} M(x) \ dx \ge \\ \frac{(2\pi)^{-n/2}\euler^{-R^2/2}}{(2n)!\normat{\Lexp M}{f}} \int_{\Omega_R} \left(f(x)\right)^{2n} \ dx.  
  \end{multline*} \qed
\end{enumerate}

\end{proof}

\section{Notable bounds and examples}
\label{sec:ineq-motiv}

There is a large body of literature about the analysis of the Gaussian space $L^2(M)$. In order to motivate our own construction and to connect it up, in this section we have collected some results about notable classes of functions that belongs to the exponential space $\Lexp M$ or to the mixture space $\LlogL M$. Some of the examples will be used in the applications of Orlicz-Sobolev spaces in the Information Geometry of the Gaussian space. Basic references on the analysis of the Gaussian space are \cite[V.1.5]{MR1335234}, \cite[4.2.1]{MR2410225}, and \cite[Ch. 1]{MR2962301}.

\subsection{Polynomial bounds}

The exponential space $\Lexp M$ contains all functions $f \in C^2(\reals^n;\reals)$ whose Hessian is uniformly dominated by a constant symmetric matrix.  In such a case, $f(x) = f(0) + \nabla f(0) x + \frac12 x^\transpose \hessianof{f(\bar x)} x$, with $x^\transpose \hessianof{f(y)} x \le \lambda \absoluteval{x}^2$, $y \in \reals^n$, and $\lambda \ge 0$ being the largest non-negative eigen-value of the dominating matrix. Then for all real $\alpha$,
 \begin{equation*}
 \int_{\reals^{n}} \euler^{\alpha f(x)} M(x) \ dx < \dfrac{1}{(2\pi)^{n/2}} \int_{\reals^{n}}  \euler^{\alpha f(0) + \nabla f(0)  x + \frac12 (\alpha \lambda -1)\absoluteval x^2} \ dx   
  \end{equation*}
and the RHS is finite for $\alpha < \lambda^{-1}$. In particular, $\Lexp M$ contains all polynomials with degree up to 2.

An interesting simple application of the same argument is the following. Assume $p = \euler^v$ is a positive density on the Gaussian space such that 
\begin{equation*}
  \euler^{A_1(x)} \le \euler^{v(x)} \le \euler^{A_2(x)}, \quad x \in \reals^n \ ,
\end{equation*}
for suitable second order polynomials $A_1$, $A_2$. Then $v \in \Lexp M$. Inequalities of this type appear in the theory of parabolic equations e.g., see \cite[Ch. 4]{MR2410225}.

The mixture space $\LlogL M$ contains all random variables $f \colon \reals^d \to \reals$ which are bounded by a polynomial, in particular, all polynomials. In fact, all polynomials belong to $L^{2}(M) \subset \LlogL M$.

\subsection{Densities of exponential form}
In this paper, we are specially interested in densities of the Gaussian space of the form $f = \euler^v$, that is $\gaussint {\euler^{v(x)}} x = 1$. Let us now consider simple properties of the mappings $f \mapsto v = \log f$ and $v \mapsto f=\euler^v$.

We have seen in Prop. \ref{prop:LlogL-entropy} that $f = \euler^v \in \LlogL M$ if, and only if,
\begin{equation*}
  - \entropyof {\euler^v} = \gaussint {\euler^{v(x)} v(x)} x < \infty \ .
\end{equation*}

As $\lim_{x\to+\infty} \frac{\cosh(x)}{x\euler^x} = 0$, we do not expect $v \in \Lexp M$ to imply $f = \euler^v \in \LlogL M$. 

As $(\cosh-1)(\alpha \log y) = (y^\alpha+y^{-\alpha})/2-1$, $\alpha > 0$, then $v  = \log f \in \Lexp M$ if, and only if, both $f^\alpha$ and $f^{-\alpha}$ both belong to $L^1(M)$ for some $\alpha > 0$. In the case $\normat {\Lexp M} v < 1$, then we can take $\alpha > 1$ and $f \in L^\alpha(M) \subset \LlogL M$. In conclusion, $\exp \colon v \mapsto \euler^v$ maps the open unit ball of $\Lexp M$ into $\cup_{\alpha>1} L^\alpha(M) \subset \LlogL M$.

This issue is discussed in the next Sec. \ref{sec:expon-manif-gauss}.

\subsection{Poincar\'e-type inequalities}
Let us denote by $C^k_b(\reals^n)$ the space of functions with derivatives up to order $k$, each bounded by a constant. We write $C^k_p(\reals^n)$ if all the derivative are bounded by a polynomial. We discuss below inequalities related to the classical Gaussian Poincar\'e inequality, which reads, in the 1-dimensional case,
\begin{equation}\label{eq:gauss-poincare}
\gaussint {\left(f(x) - \gaussint {f(y)} y \right)^2} x \le \gaussint {\avalof{f'(x)}^2} x \ ,
\end{equation}
for all $f \in C^1_p(\reals^n)$. We are going to use the same techniques used in the classical proof of \eqref{eq:gauss-poincare} e.g., see \cite{MR2962301}.

If $X$, $Y$ are independent standard Gaussian variables, then
\begin{equation*}
  X' = \euler^{-t}+\sqrt{1-\euler^{-2t}}Y, \quad Y'=\sqrt{1-\euler^{-2t}}X-\euler^{-t}Y
\end{equation*}
are independent standard Gaussian random variables for all $t \ge 0$. Because of that, it is useful to define Ornstein-Uhlenbeck semi-group by the Mehler formula
\begin{equation}\label{eq:OUsemigroup}
  P_tf(x) = \gaussint {f(\euler^{-t}x + \sqrt{1-\euler^{-2t}}y)} y, \quad t \ge 0,\quad   f \in C_p(\reals^n) \ .
\end{equation}
For any convex function $\Phi$, Jensen's inequality gives
\begin{multline*}
  \gaussint {\Phi(P_tf(x))} x \le \\ \gaussint {\gaussint {\Phi(f(\euler^{-t}x + \sqrt{1-\euler^{-2t}}y))} y } x = \\ 
  \gaussint {\Phi(f(x))} x \ .
\end{multline*}
In particular, this shows that, for all $t \ge 0$, $f \mapsto P_tf$ is a contraction for the norm of both the mixture space $\LlogL M$ and the exponential space $\Lexp M$.

Moreover, if $f \in C^1_p(\reals^n)$, we have
\begin{align}
  f(&x) - \gaussint {f(y)} y \notag \\
       &= P_0(x) - P_\infty f(x) \notag \\
       &= - \int_0^\infty \derivby t P_tf(x) \ dt \notag \\
       &= \int_0^\infty \gaussint {\nabla f(\euler^{-t}x + \sqrt{1-\euler^{-2t}}y) \cdot \left(\euler^{-t} x - \frac{\euler^{-2t}}{\sqrt{1-\euler^{-2t}}}y\right)} y \ dt \label{eq:OU-equality} \\
       &\le \int_0^\infty \frac{\euler^{-t}}{\sqrt{1-\euler^{-2t}}} \ dt \quad \times \notag \\ &\phantom{\frac{\euler^{-t}}{\sqrt{1-\euler^{-2t}}}}\gaussint {\avalof{\nabla f(\euler^{-t}x + \sqrt{1-\euler^{-2t}}y)}\avalof{\sqrt{1-\euler^{-2t}} x - \euler^{-t} y}} y \ . \label{eq:OU-inequality}
\end{align}
Note that
\begin{equation*}
  \int_0^\infty \frac{\euler^{-t}}{\sqrt{1-\euler^{-2t}}} \ dt = \int_0^1 \frac{ds}{\sqrt{1-s^2}} = \frac\pi2 \ .
\end{equation*}
We use this remark and \eqref{eq:OU-inequality} to prove our first inequality.

\begin{proposition}\label{prop:poincare-mixture}
If $f \in C^1_p(\reals^n)$ and $\lambda > 0$ is such that
\begin{equation}\label{eq:lambda}
  C\left(\lambda\frac\pi2\right) \gaussint {C(\avalof y)} y = 1 \ , \quad C(a) = \max(\avalof a,a^2) \ ,
\end{equation}
then  
  \begin{multline*}
  \int  (\cosh-1)_*\left(\lambda\left(f(x) - \int f(y) M(y) \ dy\right)\right) M(x) \ dx \le \\  \int (\cosh-1)_*(\avalof{\nabla f(x)}) M(x) \ dx \ ,
  \end{multline*}
that is
\begin{equation*}
  \normat {\LlogL M} {f - \int f(y) M(y) \ dy} \le \lambda^{-1} \normat {\LlogL M}{\avalof{\nabla f}} \ .
\end{equation*}
\end{proposition}

\begin{proof} Jensen's inequality applied to Eq. \eqref{eq:OU-inequality} gives
  \begin{multline}\label{eq:1}
    (\cosh-1)_*\left(\lambda\left(f(x) - \gaussint {f(y)} y \right)\right) \le  \int_0^\infty \frac2\pi \frac{\euler^{-t}}{\sqrt{1-\euler^{-2t}}} \ dt \ \times \\ \gaussint {(\cosh-1)_*\left(\lambda \frac\pi2 \avalof{\nabla f(\sqrt{1-\euler^{-2t}}x + \euler^{-t}y)}\avalof{\sqrt{1-\euler^{-2t}} x - \euler^{-t} y}\right)} y
  \end{multline}

Now we use  of the bound in Eq. \eqref{eq:delta2}, namely $(\cosh-1)_*(ay) \le C(a) (\cosh-1)_*(y)$ if $a >0$, where $C(a) = \max(\avalof a,a^2)$, and further bound for $a,k>0$
\begin{equation*}
C(ka) = ka \vee k^2a^2 \le kC(a) \vee k^2C(a) = C(k)C(a) \ ,  
\end{equation*}
to get
\begin{multline}\label{eq:2}
  (\cosh-1)_*\left(\lambda \frac\pi2 \avalof{\nabla f(\euler^{-t}x + \sqrt{1-\euler^{-2t}}y)}\avalof{\sqrt{1-\euler^{-2t}}x - \euler^{-t}y}\right) \le \\
    C\left(\lambda\frac\pi2\right) C\left(\avalof{\sqrt{1-\euler^{-2t}}x - \euler^{-t}y}\right) (\cosh-1)_*\left(\avalof{\nabla f(\euler^{-t}x + \sqrt{1-\euler^{-2t}}y)}\right)
 \ .
\end{multline}

Taking the expected value of both sides of the inequality resulting from \eqref{eq:1} and \eqref{eq:2}, we get
  \begin{multline*}
    \int  (\cosh-1)_*\left(\lambda\left(f(y) - \int f(x) M(x) \ dx\right)\right) M(y) \ dy \le \\
  C\left(\lambda\frac\pi2\right) \gaussint {C(\avalof{y})} y \gaussint {(\cosh-1)_*(\avalof{\nabla f(x)})} x \ ,
  \end{multline*}
We conclude by choosing a proper value of $\lambda$. \qed
\end{proof}

The same argument does not work in the exponential space. We have assume the boundedness of derivatives i.e., a Lipschitz assumption.

\begin{proposition}\label{prop:dispersion}
  If $f \in C^1_b(\reals^n)$ with $\sup\setof{\avalof{\nabla f(x)}}{x \in \reals^n} = m$ then
  \begin{equation*}
    \normat {\Lexp M} {f - \gaussint {f(y)} y} \le \frac {\pi}{2\sqrt{2\log 2}} m \ .
  \end{equation*}
\end{proposition}

\begin{proof}
Jensen's inequality applied to Eq. \eqref{eq:OU-inequality} and the assumption give
  \begin{multline*}
    (\cosh-1)\left(\lambda\left(f(x) - \gaussint {f(y)} y \right)\right) \le \\ \gaussint {(\cosh-1)\left(\lambda \frac\pi2 m x\right)} x = 
\expof{\frac{\lambda^2}2 \frac{\pi^2}4 m^2} - 1 \ .
  \end{multline*}
To conclude, choose $\lambda$ such that the the RHS equals 1. \qed   
\end{proof}

\begin{remark} Both Prop. \ref{prop:poincare-mixture} and Prop. \ref{prop:dispersion} are related with interesting results on the Gaussian space
other then bounds on norms. For example, if $f$ is a density of the Gaussian space, then the first one is a bound on the lack of uniformity $f - 1$, which, in turn, is related with the entropy of $f$. As a further example, consider a case where $\gaussint {f(x)} x = 0$ and $\normat \infty {\nabla f} < \infty$. In such a case, we have a bound on the Laplace transform of $f$, which in turn implies a bound on large deviations of the random variable $f$.
\end{remark}

To prepare the proof of an inequality for the exponential space, we start from Eq. \eqref{eq:OU-equality} and observe that for $f \in C^2_p(\reals^n)$ we can write
\begin{multline*}
  f(x) - \gaussint {f(y)} y = \\  
      \int_0^{\infty} \euler^{-t} \left(\gaussint {\nabla f(\euler^{-t}x+\sqrt{1-\euler^{-2t}}y)} y\right) \cdot x \ dt \\
  - \int_0^{\infty} \euler^{-2t} \gaussint {\nabla \cdot \nabla f(\euler^{-t}x+\sqrt{1-\euler^{-2t}}y)} y \ dt \ ,
\end{multline*}
where integration by parts and $(\partial/\partial y_i) M(y) = - y_i M(y)$ have been used to get the last term.

If we write $f_i(z) = \partiald {z_i}$ and $f_{ii}(z) = \partialdd {z_i} f(z)$ then
\begin{equation*}
  \partiald {x_i} P_tf(x) = \euler^{-t} P_tf_i(x)
\end{equation*}
and
\begin{equation*}
  \partialdd {x_i} P_tf(x) = \euler^{-2t} P_tf_{ii}(x) \ ,
\end{equation*}
so that
\begin{equation*}
  f(x) - \gaussint {f(y)} y = 
 \int_0^{\infty} \left(x \cdot \nabla P_tf(x) - \nabla \cdot \nabla P_tf(x)\right) \ dt \ .
\end{equation*}

If $g \in C^2_b(\reals^n)$ we have
\begin{multline}\label{eq:OU-selfadjoint}
  \gaussint {g(x) \left(f(x) - \gaussint {f(y)} y\right)} x = \\
  \int_0^\infty \left(\gaussint {g(x) x \cdot \nabla P_tf(x)} x - \gaussint {g(x) \nabla \cdot \nabla P_tf(x)} x \right) \ dt = \\ \int_0^\infty \left(\gaussint {g(x) x \cdot \nabla P_tf(x)} x + \int \nabla(g(x)M(x)) \cdot \nabla P_tf(x) \ dx \right) \ dt = \\ \int_0^\infty \gaussint {\nabla g(x) \cdot \nabla P_tf(x)} x \ dt =  \\ \int_0^\infty \euler^{-t} \gaussint {\nabla g(x) \cdot P_t \nabla f(x)} x \ dt \ . 
 \end{multline}

  Let $\avalof \cdot _1$ and $\avalof \cdot _2$ be two norms on $\reals^n$ such that $\avalof{x \cdot y} \le \avalof x _1 \avalof y _2$. Define the covariance of $f,g \in C^2_p(\reals^n)$ to be
  \begin{multline*}
    \covat M f g = \\ \gaussint{\left(f(x) - \gaussint {f(y)} y\right)g(x)} x = \\ \gaussint{\left(f(x) - \gaussint {f(y)} y\right)\left(g(x) - \gaussint {g(y)} y\right)} x \ .
  \end{multline*}

\begin{proposition}
If $f,g \in C^2_p(\reals^n)$, then
\begin{equation*}
  \avalof{\covat M  f g} \le \avalof {\normat {\LlogL M} {\nabla f}} _1 \avalof{\normat {(\LlogL M)^*} {\nabla g} } _2 \ .  
\end{equation*}
\end{proposition}

\begin{proof}
We use Eq. \eqref{eq:OU-selfadjoint} and the inequality \eqref{eq:orlicznorm}.
\begin{multline*}
  \avalof{ \gaussint {\nabla g(x) \cdot P_t \nabla f(x)} x} \le \\
  \sum_{i=1}^n \avalof{\gaussint {g_i(x) P_t f_i(x)} x} \le \\
  \sum_{i=1}^n \normat {\LlogL M^*} {g_i} \normat {\LlogL M} {P_t f_i} \le  \\
  \sum_{i=1}^n \normat {\LlogL M^*} {g_i} \normat {\LlogL M} {f_i} \le  \\
  \avalof {\normat {\LlogL M^*} {\nabla g}} _1 \avalof{\normat {\LlogL M} {\nabla f}} _2 \ . 
\end{multline*}
\qed
\end{proof}

If $g_n$ is a sequence such that $\nabla g_n \to 0$ in $\Lexp M$, then the inequality above shows that $g_n - \gaussint {g_n(x)} x \to 0$.

\section{Exponential manifold on the Gaussian space}
\label{sec:expon-manif-gauss}

In this section we first review the basic features of our construction of IG as it was discussed in the original paper \cite{MR1370295}. Second, we see how the choice of the Gaussian space adds new features, see \cite{MR3126029,MR3365132}. We normally use capital letters to denote random variables and write $\expectat M U = \gaussint {U(x)} x$.

We define $\Bspace M = \setof{U \in \Lexp M}{\expectat M U = 0}$. The positive densities of the Gaussian space we consider are all of the exponential form $p = \euler^{U}/Z_M(U)$, with $U \in \Bspace M \subset \Lexp M$ and normalization (moment functional, partition functional) $Z_M(U) = \expectat M U < \infty$.

We can also write $p = \euler^{U-K_K(U)}$, where $K_M(U) = \log Z_M(U)$ is called cumulant functional. Because of the assumption $\expectat M U = 0$, the chart mapping

\begin{equation*}
s_M \colon  p \mapsto \log p - \expectat M {\log p} = U 
\end{equation*}
is well defined.

Both the extended real functions $Z_M$ and $K_M$ are convex on $\Bspace M$. The common proper domain of $Z_M$ and of $K_M$ contains the open unit ball of $\Bspace M$.  In fact, if $\expectat M {(\cosh-1)(\alpha U)} \le 1$, $\alpha > 1$, then, in particular, $Z_M(U) \le 4$.

We denote $\sdomain M$ the interior of the proper domain of the cumulant functional. The set $\sdomain M$ is nonempty, convex, star-shaped, and solid i.e., the generated vector space is $\Bspace M$ itself.

We define the maximal exponential model to be the set of densities on the Gaussian space $\maxexp M = \setof{\euler^{U - K_M(U)}}{U \in \sdomain M}$. 

We prove below that the mapping $e_M = s^{-1}_M \colon \sdomain M \to \maxexp M$ is smooth. The chart mapping itself $s_M$ is not and induces on $\maxexp M$ a topology that we do not discuss here.

\begin{proposition}\label{prop:SMtoLlogL}
  The mapping $e_M \colon \sdomain M \ni U \mapsto \euler^{U-K_M(U)}$ is continuously differentiable in $\LlogL M$ with derivative $d_H e_M(U) = e_M(U) (U - \expectat {e_M(U)\cdot M} {H})$.
\end{proposition}

\begin{proof} We split the proof into numbered steps.
  \begin{enumerate} 
  \item If $U \in \sdomain M$ then $\alpha U \in \sdomain M$ for some $\alpha > 1$. Moreover, $(\cosh-1)_*(y) \le C(\alpha) \avalof y^\alpha$. Then
\begin{equation*}
  \expectat M {(\cosh-1)_*(e^U)} \le \text{const} \ \expectat M {\euler^{\alpha U}} < \infty \ .
\end{equation*}
so that $\euler^{U} \in \LlogL M$. It follows that $e_M(U) = \euler^{U-K_M(U)} \in \LlogL M$.
\item
Given $U \in \sdomain M$, as $\sdomain M$ is open in the exponential space $\Lexp M$, there exists a constant $\rho > 0$ such that $\normat {\Lexp M} H \le \rho$ implies $U + H \in \sdomain M$. In particular,
\begin{equation*}
  U + \frac{\rho}{\normat{\Lexp M} U} U = \frac{\normat{\Lexp M} U+ \rho}{\normat{\Lexp M} U} U \in \sdomain M \ .
\end{equation*}

We have, from the H\"older's inequality with conjugate exponents
\begin{equation*}
  \frac{2(\normat{\Lexp M} U+\rho)}{2\normat{\Lexp M} U+\rho}, \quad  \frac{2(\normat{\Lexp M} U+\rho)}{\rho} \ ,
\end{equation*}
that
\begin{multline*}
  \expectat M {\expof{\frac{2\normat{\Lexp M} U+\rho}{2\normat{\Lexp M} U}(U+H)}} \le \\
  \expectat M {\expof{\frac{\normat{\Lexp M} U + \rho}{\normat{\Lexp M} U}}U}^{\frac{2\normat{\Lexp M} U+\rho}{2(\normat{\Lexp M} U+\rho)}} \quad \times \\
  \expectat M {\expof{\frac{(2\normat{\Lexp M} U + \rho)(\normat{\Lexp M} U+\rho)}{\rho\normat{\Lexp M} U}H}}^{\frac{\rho}{2(\normat{\Lexp M} U+\rho)}} \ .
\end{multline*}
In the RHS, the first factor is finite because the random variable under $\exp$ belong to $\sdomain M$, while the second factor is bounded by a fixed constant for all $H$ such that 
\begin{equation*}
  \normat {\Lexp M} H \le \frac{\rho\normat{\Lexp M} U}{(2\normat{\Lexp M} U + \rho)(\normat{\Lexp M} U+\rho)} \ .
\end{equation*}
This shows that $e_M$ is locally bounded in $\Lexp M$.
\item
Let us now consider $\prod_{i=1}^m H_i\euler^U$ with $\normat {\Lexp M}{H_i} \le 1$ for $i=1,\dots,m$ and $U \in \sdomain M$. Chose an $\alpha > 1$ such that $\alpha U \in \sdomain M$, and observe that, because of the previous item applied to $\alpha U$, the mapping $U \mapsto \expectat M {\euler^{\alpha U}}$ is uniformly bounded in a neighborhood of $U$ by a constant $C(U)$. As $\alpha > (\alpha+1)/2 > 1$ and we have the inequality $(\cosh-a)_*(y) \le \frac{C((1+\alpha)/2)}{(1+\alpha)/2} \avalof y^{(1+\alpha)/2}$. It follows, using the $(m+1)$-terms Fenchel-Young inequality for conjugate exponents $2\alpha/(\alpha+1)$ and $2m\alpha/(\alpha-1)$ ($m$ times), that
  \begin{multline*}
    \expectat M {(\cosh-1)_*\left(\prod_{i=1}^mH_i\euler^U\right)} \le \\  \frac{C((1+\alpha)/2)}{(1+\alpha)/2} \expectat M {\prod_{i=1}^m \avalof H ^{(1+\alpha)/2} \euler^{(1+\alpha)U/2}} \le \\ \frac{C((1+\alpha)/2)}{(1+\alpha)/2} \left(\expectat M {\euler^{\alpha U}} + \sum_{i=1}^m \expectat M {\avalof {H_i}^{m\alpha(1+\alpha)/(\alpha-1)}}\right)\le \\ \frac{C((1+\alpha)/2)}{(1+\alpha)/2}\left(C(U) + \sum_{i=1}^m \normat {L^{m\alpha(1+\alpha)/(\alpha-1)}(M)} H ^{m\alpha(1+\alpha)/(\alpha-1)}\right) \ ,
  \end{multline*}
which is bounded by a constant depending on $U$ and $\alpha$. We have proved  that the multi-linear mapping $(H_1,\dots,H_m) \mapsto \prod_{i=1}^m H_i \euler^U$ is continuous from $(\Lexp M)^m$ to $\LlogL M$, uniformly in a neighborhood of $U$.
\item
  Let us consider now the differentiability of $U \mapsto \euler^U$. For $U+H \in \sdomain M$, it holds
  \begin{multline*}
  0 \le  \euler(U+H) - \euler^U - \euler^U H = \int_0^1 (1-s) \euler^{U + sH} H^2 \ ds = \\ \int_0^1 (1-s) \euler^{(1-s) U + s(U+H)} H^2 \ ds \le \\ \int_0^1 (1-s)^2 \euler^U H^2 \ ds + \int_0^1 s(1-s) \euler^{U+H} H^2 \ ds = \\ \left(\frac13 e^U + \frac16 \euler^{U+H} \right) H^2 \ .
\end{multline*}
Because of the previous item, the RHS is bounded by a constant times $\normat {\Lexp M} H ^2$ for $\normat {\Lexp M} H$ small, which in turn implies the differentiability. Note that the bound is uniform in a neighborhood of $U$. 
\item It follows that $Z_M$ and $K_M$ are differentiable and also $e_M$ is differentiable with locally uniformly continuous derivative.
\end{enumerate}
\qed
\end{proof}

We turn to discuss the approximation with smooth random variables. We recall that $(\cosh-1)_*$ satisfies the $\Delta_2$-bound
\begin{equation*}
  (\cosh-1)_*(ay) \le \max(\avalof a,a^2)(\cosh-1)_*(y)
\end{equation*}
hence, bounded convergence holds for the mixture space $\LlogL M$. That, in turn, implies separability. This is not true for the exponential space $\Lexp M$. Consider for example $f(x) = \absoluteval x^2$. This function belongs in $\Lexp M$, but, if $f_R(x) = f(x)(\avalof x \ge R)$, then
\begin{equation*}
  \int (\cosh-1)(\epsilon^{-1} f_R(x)) \ M(x)dx \ge \frac12 \int_{\absoluteval x > R} \euler^{\epsilon^{-1} \absoluteval x^2} \ M(x)dx = +\infty, \quad \text{if $\epsilon \le 2$} \ ,
\end{equation*}
hence there is no convergence to 0. However, the truncation of $f(x) = \absoluteval x$ does converge.

While the exponential space $\Lexp M$ is not separable nor reflexive, we have the following weak property.  Let $\Ccs{\reals^n}$ and $\Cinfcs{\reals^n}$ respectively denote the space of continuous real functions with compact support and the space of infinitely-differentiable real functions on $\reals^n$ with compact support. The following proposition was stated in \cite[Prop. 2]{pistone:2017-GSI2017}.

\begin{proposition}\label{prop:density} 
For each $f \in \Lexp M$ there exist a nonnegative function $h \in \Lexp M$ and a sequence $g_n \in \Cinfcs{\reals^n}$ with $\absoluteval{g_n} \le h$, $n=1,2,\dots$, such that $\lim_{n\to\infty} g_n = f$ a.e. As a consequence, $\Cinfcs {\reals^n}$ is weakly dense in $\Lexp M$. 
\end{proposition}
\begin{proof} Our proof uses a monotone class argument \cite[Ch. II]{dellacherie|meyer:75}. Let $\mathcal H$ be the set of all random variables $f \in \Lexp M$ for which there exists a sequence $g_n \in \Ccs{\reals^n}$ such that $g_n(x) \to f(x)$ a.s. and $\absoluteval {g_n(x)} \le \absoluteval{f(x)}$. Let us show that $\mathcal H$ is closed for monotone point-wise limits of positive random variables. Assume $f_n \uparrow f$ and $g_{n,k} \to f_n$ a.s. with $\absoluteval{g_{n,k}} \le f_n \le f$. Each sequence $(g_{nk})_k$ is convergent in $L^1(M)$ then, for each $n$ we can choose a $g_n$ in the sequence such that $\normat {L^1(M)} {f_n - g_n} \leq 2^{-n}$. It follows that $\absoluteval{f_n - g_n} \to 0$ a.s. and also $f - g_n = (f-f_n)+(f_n-g_n) \to 0$ a.s. Now we can apply the monotone class argument to $\Ccs{\reals^n} \subset \mathcal H$. The conclusion follows from the uniform density of $\Cinfcs {\reals^n}$ in $\Ccs{\reals^n}$. \qed 
\end{proof}

The point-wise bounded convergence of the previous proposition implies a result of local approximation in variation of finite-dimensional exponential families.

\begin{proposition}
Given $U_1,\dots, U_m \in \Bspace M$, consider the exponential family
\begin{equation*}
  p_\theta = \expof{\sum_{j=1}^m \theta_j U_j - \psi(\theta)}, \quad \theta \in \Theta \ .
\end{equation*}
There exists a sequence $(U_1^k,\dots,U_m^k)_{k\in\naturals}$ in $\Cinfcs {\reals^n}^m$ and an $\alpha > 0$ such that the sequence of exponential families
\begin{equation*}
  p^k_\theta = \expof{\sum_{j=1}^m \theta_j U^k_j - \psi_k(\theta)}
\end{equation*}
is convergent in variation to $p_\theta$ for all $\theta$ such that $\sum\avalof{\theta_j} < \alpha$.
\end{proposition}

\begin{proof}
  For each $j=1,\dots,m$ there exists a point-wise converging sequence $(U_j^k)_{k\in\naturals}$ in $\Cinfcs{\reals^n}$ and a bound $h_i \in \Lexp M$. define $h = \wedge_{j=1,\dots,m} h_j$.  Let $\alpha > 0$ be such that $\expectat M {(\cosh-1)(\alpha h)} \le 1$, which in turn implies $\expectat M {\euler^{\alpha h}} \le 4$. Each $\sum_j \theta_j U_J^k$ is bounded in absolute value by $\alpha h$ if $\sum_{j=1}^m \avalof {\theta_j} < \alpha$.

As
\begin{equation*}
  \psi(\theta) = K_M(\sum_{j=1}^m \theta_j U_j) = \log \expectat M {\euler^{\sum_{j=1}^m \theta_j U_j}}
\end{equation*}
and
\begin{equation*}
  \psi_k(\theta) = \log \expectat M {\euler^{\sum_{j=1}^m \theta_j U^k_j}}
\end{equation*}
dominated convergence implies $\psi_k(\theta) \to \psi(\theta)$ and hence $p_k(x;\theta) \to p(x;\theta)$ for all $x$ if $\sum_{j=1}^m \avalof {\theta_j} < \alpha$. Sheff\'e lemma concludes the proof. \qed
\end{proof}

\subsection{Maximal exponential manifold as an affine manifold} \label{sec:affine}
The maximal exponential model $\maxexp M = \setof{\euler^{U - K_M(U)}}{U \in \Bspace M}$ is an elementary manifold embedded into $\LlogL M$ by the smooth mapping $e_M \colon \sdomain M \to \LlogL M$. There is actually an atlas of charts that makes it into an affine manifold, see \cite{MR3126029}. We discuss here some preliminary results about this important topic.

An elementary computation shows that
\begin{equation}\label{eq:simple}
  (\cosh-1)^2(u) = \frac12(\cosh-1)(2u) - 2 (\cosh-1)(u) \le \frac 12 (\cosh-1)(2u) 
\end{equation}
and, iterating,
\begin{equation*}
  (\cosh-1)^{2k}(u) \le \frac 1{2^k} (\cosh-1)(2^ku) \ .   
\end{equation*}

\begin{proposition}\label{prop:tangentspaces}
  If $f,g\in \maxexp M$, then $\Lexp {f\cdot M} = \Lexp {g\cdot M}$.
\end{proposition}

\begin{proof}
Given any $f \in \maxexp M$, with $f = \euler^{U-K_M(U)}$ and $U \in \sdomain M$, and any $V \in \Lexp M$, we have from Fenchel-Young inequality and Eq. \eqref{eq:simple} that
\begin{multline*}
 \gaussint {(\cosh-1)(\alpha V(x)) f(x)} x \le \\ \frac1{2^{k+1}k}\gaussint {(\cosh-1)(2k\alpha V(x))} x + \\ \frac{2k-1}{2k} Z_M(U)^{\frac{2k}{2k-1}}\gaussint {\expof{\frac{2k}{2k-1}U}} x \ .
\end{multline*}
If $k$ is such that $\frac{2k}{2k-1}U \in \sdomain M$, one sees that $V \in \Lexp {f\cdot M}$. We have proved that $\Lexp M \subset \Lexp{f \cdot M}$.

Conversely, 
\begin{multline*}
 \gaussint {(\cosh-1)(\alpha V(x))} x = \gaussint {(\cosh-1)(\alpha V(x)) f^{-1}(x)f(x)} x \le \\ \frac1{2^{k+1}k}\gaussint {(\cosh-1)(2k\alpha V(x))f(x)} x + \\ \frac{2k-1}{2k} Z_M(U)^{\frac{1}{2k-1}}\gaussint {\expof{\frac{1}{2k-1}U}} x \ .
\end{multline*}
If $\frac1{2k-1} U \in \sdomain M$, one sees that $V \in \Lexp {f\cdot M}$ implies $V \in \Lexp M$. \qed
\end{proof}

The affine manifold is defined as follows. For each $f \in \maxexp M$, we define the Banach space
\begin{equation*}
  \Bspace f = \setof{U \in \Lexp {f\cdot M}}{ \expectat {f\cdot M} U = 0} = \setof{U \in \Lexp {M}}{ \expectat {M} {Uf} = 0} \ , 
\end{equation*}
and the chart
\begin{equation*}
  s_f \colon \maxexp M \ni g \mapsto \log\frac gf - \expectat {f\cdot M} {\log\frac gf} \ .
\end{equation*}

It is easy to verify the following statement, which defines the \emph{exponential affine manifold}. Specific properties related with the Gaussian space are discussed in the next Sec. \ref{sec:mollifiers} and space derivatives in Sec. \ref{sec:orlicz-sobolev}.

\begin{proposition}\label{prop:affinemanifold}
  The set of charts $s_f \colon \maxexp M \to \Bspace f$ is an affine atlas of global charts on $\maxexp M$. 
\end{proposition}

On each fiber $S_p\maxexp M = B_p$ of the statistical bundle the covariance $(U,V) \mapsto \expectat M {UV} = \scalarat p U V$ provides a natural metric. In that metric the natural gradient of a smooth function $F \colon \maxexp M \to \reals$ is defined by 
\begin{equation*}
  \derivby t F(p(t)) = \scalarat {p(t)} {\Grad F(p(t))}{Dp(t)} \ ,
\end{equation*}
where $t \mapsto p(t)$ is a smooth curve in $\maxexp M$ and $Dp(t) = \derivby t \log p(t)$ is the expression of the velocity.

\subsection{Translations and mollifiers}
\label{sec:mollifiers}

In this section, we start to discuss properties of the exponential affine manifold of Prop. \ref{prop:affinemanifold} which depend on the choice of the Gaussian space as base probability space.

Because of the lack of norm density of the space of infinitely differentiable functions with compact support $\Cinfcs {\reals^n}$ in the exponential space $\Lexp M$, we introduce the following classical the definition of Orlicz class.

\begin{definition}
We define the \emph{exponential class}, $\Cexp M$, to be the closure of $\Ccs{\reals^n}$ in the exponential space $\Lexp M$.  
\end{definition}

We recall below the characterization of the exponential class.

\begin{proposition}\label{prop:d}
 Assume $f \in \Lexp M$ and write $f_R(x) = f(x)(\absoluteval x > R)
$. The following conditions are equivalent: 
\begin{enumerate}
\item\label{item:d1}  The real function $\rho \mapsto \int (\cosh-1)(\rho f(x)) \ M(x)dx$ is finite for all $\rho > 0$.
\item\label{item:d3}  $f \in C^{\cosh-1}(M)$.
\item\label{item:d2}  $\lim_{R \to \infty} \normat {\Lexp M} {f_R} = 0$.
\end{enumerate}
\end{proposition}

\begin{proof} This is well known e.g., see \cite[Ch. II]{MR724434}. A short proof is given in our note \cite[Prop. 3]{pistone:2017-GSI2017}. \qed
\end{proof}

Here we study of the action of translation operator on the exponential space $\Lexp M$ and on the exponential class $\Cexp M$. We consider both translation by a vector, $\tau_h f(x) = f(x - h)$, $h\in\reals^n$, and translation by a probability measure, of convolution, $\mu$, $\tau_\mu f(x) = \int f(x-y) \ \mu(dy) = f*\mu(x)$. A small part of this material was published in the conference paper \cite[Prop. 4--5]{pistone:2017-GSI2017}.

\begin{proposition}[Translation by a vector]\
\label{prop:taubyh}  
\begin{enumerate}
\item \label{item:taubyh1} For each $h \in \reals^n$, the translation mapping $\Lexp M \ni f \mapsto \tau_h f$ is linear and bounded from $\Lexp M$ to itself. In particular,
  \begin{equation*}
    \normat {\Lexp M} {\tau_h f} \le 2 \normat {\Lexp M} f \quad \text{if} \quad \absoluteval h \le \sqrt{\log 2} \ .
  \end{equation*}

\item \label{item:taubyh3} For all $g \in \LlogL M$ we have
  \begin{equation*}
    \scalarat M {\tau_h f} g = \scalarat M f {\tau_h^* g}, \quad \tau_h^* g(x) = \euler^{-h \cdot x - \frac12 \absoluteval h^2} \tau_{-h}g(x) \ ,
  \end{equation*}
and $\absoluteval h \le \sqrt {\log 2}$ implies $\normat {\Lexp M)^*} {\tau^*_h g} \le 2 \normat {\Lexp M)^*} g$. The translation mapping $h \mapsto \tau_h^*g$ is continuous in $\LlogL M$.

\item \label{item:taubyh2} If $f \in \Cexp M$ then $\tau_h f \in \Cexp M$, $h \in \reals^n$, and the mapping $\reals^n \colon h \mapsto \tau_h f$ is continuous in $\Lexp M$. 
  \end{enumerate}
\end{proposition}

    \begin{proof}
    \begin{enumerate}

\item Let us first prove that $\tau_h f \in \Lexp M$. Assume $\normat {\Lexp M} f \le 1$. For each $\rho > 0$, writing $\Phi = \cosh-1$,
      \begin{multline*}
        \int \Phi(\rho \tau_h f(x))\ M(x)dx = \int \Phi(\rho f(x-h))\ M(x)dx = \\ \int \Phi(\rho f(z))\ M(z+h)dz = \euler^{ - \frac12 \absoluteval h^2} \int \euler^{- z\cdot h} \Phi(\rho f(z))\  M(z)dz \ ,
      \end{multline*}
hence, using H\"older inequality and the inequality in Eq. \eqref{eq:simple},
\begin{multline}\label{eq:doublerho}
  \int \Phi(\rho\tau_h f(x)) \ M(x)dx \le \\ \euler^{ - \frac12 \absoluteval h^2} \left(\int \euler^{- 2z\cdot h}\ M(z)dz\right)^{\frac12} \left(\int \Phi^2(\rho f(z))\ M(z)dz\right)^{\frac12} \le \\ \frac1{\sqrt 2}\euler^{\frac{\absoluteval h^2}2} \left(\int \Phi(2\rho f(z)) M(z)\ dz\right)^{\frac12} \ .
\end{multline}
Take $\rho=1/2$, so that $\expectat M {\Phi\left(\tau_h \frac12 f(x)\right)} \le \frac1{\sqrt 2}\euler^{\frac{\absoluteval h^2}2}$, which implies $f \in \Lexp M$. Moreover, $\normat {\Lexp M}{\tau_h f} \le 2$ if $\frac1{\sqrt 2}\euler^{\frac{\absoluteval h^2}2} \le 1$.

The semi-group property $\tau_{h_1+h_2} f = \tau_{h_1} \tau_{h_2} f$ implies the boundedness for all $h$.

\item The computation of $\tau^*_h$ is
\begin{align*}
  \scalarat M {\tau_h f} g &= \int f(x-h)g(x) \ M(x)dx \\
&= \int f(x) g(x+h)M(x+h) \ dx \\
&= \int f(x) \euler^{-h\cdot x - \frac12\absoluteval h^2} \tau_{-h}g(x) \ M(x)dx \\
&= \scalarat M f {\tau^*_h g} \ .
\end{align*}
Computing Orlicz norm of the mixture space, we find
\begin{multline*}
\normat {(\Lexp M)^*} {\tau^*_h g} = \sup\setof{\scalarat M f {\tau^*_h g}}{\normat {\Lexp M} f \le 1}  = \\ \sup\setof{\scalarat M {\tau_h f} g}{\normat {\Lexp M} f \le 1} \ .
\end{multline*}
From the previous item we know that $\absoluteval h \le \sqrt {\log 2}$ implies
\begin{multline*}
  \scalarat M {\tau_h f} g \le \normat {\Lexp M} {\tau_h f} \normat {(\Lexp M)^*} g \le \\ 2 \normat {\Lexp M} {f} \normat {(\Lexp M)^*} g \ , 
\end{multline*}
hence $\normat {(\Lexp M)^*} {\tau^*_h g} \le  2 \normat {\Lexp M)^*} g$.

Consider first the continuity a 0. We have for $\avalof h \le \sqrt{\log 2}$ and any $\phi \in \Cinfcs{\reals^n}$ that
\begin{multline*}
  \normat{(\Lexp M)^*}{\tau_hg - g} \le \\ \normat{(\Lexp M)^*}{\tau_h (g-\phi)} + \normat{(\Lexp M)^*}{\tau_h \phi - \phi} + \normat{(\Lexp M)^*}{\phi - g} \le \\ 3 \normat{(\Lexp M)^*}{g - \phi} + \sqrt 2 \normat {\infty}{\tau_h \phi - \phi} \ .  
\end{multline*}
The first term in the RHS is arbitrary small because of the density of $\Cinfcs{\reals^n}$ in $\LlogL M$, while the second term goes to zero as $h \to 0$ for each $\phi$.

The general case follows from the boundedness and the semi-group property.

\item If $f \in \Cexp M$, then , by Prop. \ref{prop:d}, the RHS of Eq. \eqref{eq:doublerho} is finite for all $\rho$, which in turn implies that $\tau_h f \in \Cexp M$ because of Prop. \ref{prop:d}.1. Other values of $h$ are obtained by the semi-group property.

  The continuity follows from the approximation argument, as in the previous item.
       \end{enumerate}
\qed  \end{proof}

  We denote by $\probabilities$ the convex set of probability measures on $\reals^n$ and call \emph{weak convergence} the convergence of sequences in the duality with $C_b(\reals^n)$. In the following proposition we denote by $\probabilities_{\euler}$ the set of probability measures $\mu$ such that $h \mapsto \euler^{\frac12 \absoluteval h^2}$ is integrable. For example, this is the case when $\mu$ is Gaussian with variance $\sigma^2I$, $\sigma^2 < 1$, or when $\mu$ has a bounded support. Weak convergence in $\probabilities_e$ means $\mu_n \to \mu$ weakly and $\int \euler^{\frac12 \absoluteval h^2} \ \mu_n(dh) \to \int \euler^{\frac12 \absoluteval h^2} \ \mu(dh)$.  Note that we study here convolutions for the limited purpose of deriving the existence of smooth approximations obtained by \emph{mollifiers}, see \cite[108--109]{MR2759829}.

\begin{proposition}[Translation by a probability]
\label{prop:taubymu} Let $\mu \in \probabilities_{\euler}$.   
\begin{enumerate}
\item \label{item:taubymu1}The mapping $f \mapsto \tau_\mu f$ is linear and bounded from $\Lexp M$ to itself. If, moreover, $\int \euler^{\frac12\avalof{ h^2}}\ \mu(dh) \le \sqrt {2}$, then $\normat {\Lexp M} {\tau_\mu f} \le 2 \normat {\Lexp M} f$.
\item \label{item:taubymu2} If $f \in \Cexp M$ then $\tau_\mu f \in \Cexp M$. The mapping $\probabilities \colon \mu \mapsto \tau_\mu f$ is continuous at $\delta_0$ from the weak convergence to the $\Lexp M$ norm.
  \end{enumerate}
\end{proposition}
    \begin{proof} \ 
    \begin{enumerate}
    \item Let us write $\Phi=\cosh-1$ and note the Jensen's inequality
      \begin{multline*}
        \Phi\left(\rho \tau_\mu f(x)\right) = \Phi\left(\rho \int f(x-h) \ \mu(dh)\right) \le \\
\int \Phi\left(\rho f(x-h)\right) \ \mu(dh) = \int \Phi\left(\rho \tau_h f(x)\right) \ \mu(dh) \ .
      \end{multline*}
By taking the Gaussian expectation of the previous inequality we have, as in the previous item,
\begin{multline}\label{eq:mutranslation}
  \expectat M {\Phi\left(\rho \tau_\mu f\right)} \le 
  \int \gaussint {\Phi\left(\rho f(x-h)\right)} x \ \mu(dh) = \\
  \int \euler^{-\frac12\avalof h^2}\gaussint {\euler^{-h\cdot z} \Phi\left(\rho f(z)\right)} z \ \mu(dh) \le \\
  \frac 1{\sqrt 2} \int \euler^{\frac12\avalof h^2} \ \mu(dh) \ \expectat M {\Phi(2\rho f)} \ .
\end{multline}
If $\normat {\Lexp M} f \le 1$ and $\rho=1/2$, the RHS is bounded, hence $\tau_\mu f \in \Lexp M$. If, moreover, $\int \euler^{\frac12\avalof h^2} \le \sqrt{2}$, then the RHS is bounded by 1, hence $\normat {\Lexp M} {\tau_\mu f} \le 2$.

\item We have found above that for each $\rho > 0$ it holds \eqref{eq:mutranslation}, where the right-end-side if finite for all $\rho$ under the current assumption. It follows from Prop. \ref{prop:d} that $\tau_h f \in \Cexp M$.

To prove the continuity at $\delta_0$, assume $\int \euler^{\frac12\absoluteval h^2} \ \mu(dh) \le \sqrt 2$, which is always feasible if $\mu \to \delta_0$ in $\probabilities_\euler$ weakly. Given $\epsilon > 0$, choose $\phi \in \Cinfcs {\reals^n}$ so that $\normat {\Lexp M} {f - \phi} < \epsilon$. We have
  \begin{multline*}
    \normat {\Lexp M} {\tau_\mu f - f} \le \\ \normat {\Lexp M} {\tau_\mu(f-\phi)} + \normat {\Lexp M} {\tau_\mu \phi - \phi} + \normat {\Lexp M} {\phi - f} \le \\ 3\epsilon + A^{-1} \normat {\infty} {\tau_\mu \phi - \phi} \ ,
  \end{multline*}
where $A = \normat {\Lexp M} 1$. As $\lim_{\mu \to \delta_0} \normat {\infty} {\tau_\mu \phi - \phi} = 0$, see e.g. \cite[III-1.9]{MR1335234}, the conclusion follows.

     \end{enumerate}
\qed  \end{proof}

We use the previous propositions to show the existence of smooth approximations through sequences of mollifiers. A \emph{bump} function is a non-negative function $\omega$ in $C_0^\infty(\reals^n)$ such that $\int \omega(x) \ dx = 1$. It follows that $\int \lambda^{-n} \omega(\lambda^{-1} x) \ dx = 1$, $\lambda > 0$ and the family of mollifiers $\omega_\lambda(dx) = \lambda^{-n} \omega(\lambda^{-1} x)dx$, $\lambda > 0$, converges weakly to the Dirac mass at 0 as $\lambda \downarrow 0$ in $\probabilities_\euler$. Without restriction of generality, we shall assume that the support of $\omega$ is contained in $[-1,+1]^n$.

For each $f \in \Lexp M$ we have
\begin{equation*}
\tau_{\omega_\lambda}(x) =   f \ast \omega_\lambda (x) = \int f(x-y)\lambda^{-n}\omega(\lambda^{-1}y) \ dy = \int_{[-1,+1]^n} f(x - \lambda z)\omega(z) \ dz \ .
\end{equation*}

For each $\Phi$ convex we have by Jensen's inequality that 
\begin{equation*}
  \Phi\left(f \ast \omega_\lambda (x)\right) \le (\Phi\circ f) \ast \omega_\lambda (x)
\end{equation*}
and also
\begin{multline*}
  \int \Phi\left(f \ast \omega_\lambda (x)\right) M(x) \ dx \le \int \int_{[-1,+1]^n} \Phi \circ f(x - \lambda z)\omega(z) \ dz M(x) \ dx = \\
  \int\Phi\circ f(y) \left(\int_{[-1,+1]^n} \expof{-\lambda \scalarof z y - \frac{\lambda^2}2 \absoluteval z^2}\omega(z) \ dz\right) M(y) \ dy \le \\ \int\Phi\circ f(y)  M(y) \ dy \ . 
\end{multline*}
\begin{proposition}[Mollifiers]\label{prop:molli} Let be given a family of mollifiers $\omega_\lambda$, $\lambda > 0$.
For each $f \in \Cexp M$ and for each $\lambda > 0$ the function 
\begin{equation*}
\tau_{\omega_\lambda}f(x) = \int f(x-y) \lambda^{-n} \omega(\lambda^{-1}y) \ dy = f\ast\omega_\lambda(x)
\end{equation*}
belongs to $C^\infty(\reals^n)$. Moreover,
\begin{equation*}
  \lim_{\lambda \to 0} \normat  {\Lexp M} {f\ast\omega_\lambda - f} = 0 \ . 
\end{equation*}
\end{proposition}

\begin{proof}
  Any function in $\Lexp M$ belongs to $L^1_{\text{loc}}(\reals^n)$, hence
\begin{equation*}
x \mapsto \int f(x-y) \omega_\lambda(y)dy = \int f(z) \omega_\lambda(z-x)dz
\end{equation*}
belongs to $C^\infty(\reals^n)$, see e.g. \cite[Ch. 4]{MR2759829}. Note that $\int \euler^{\absoluteval h^2/2} \omega_\lambda(dh) < + \infty$ and then apply Prop. \ref{prop:taubymu}\eqref{item:taubymu2}. 

\end{proof}

\begin{remark}\ 
  Properties of weighted Orlicz spaces with the $\Delta_2$-property can be sometimes deduced from the properties on the un-weighted spaces by suitable embeddings, but this is not the case for the exponential space. Here are two examples.
\begin{enumerate}
  \item Let $1 \le a < \infty$. The mapping $g \mapsto gM^{\frac1a}$ is an isometry of $L^a(M)$ onto $L^a(\reals^n)$. As a consequence, for each $f \in L^1(\reals^n)$ and each $g \in L^a(M)$ we have $\normat {L^a(M)} {\left[f*(gM^{\frac1a})\right]M^{-\frac1a}} \le \normat {L^1(\reals^n)} f \normat{L^a(M)}{g}$. 
\item The mapping
  \begin{equation*}
    g \mapsto \signof{g}(\cosh - 1)_*^{-1}(M(\cosh-1)_*(g))
  \end{equation*}
is a surjection of $\LlogL{\reals^n}$ onto $\LlogL M$ with inverse
\begin{equation*}
  h \mapsto \signof{h}(\cosh - 1)_*^{-1}(M^{-1}(\cosh-1)_*(f)) \ .
\end{equation*}
It is surjective from unit vectors (for the Luxemburg norm) onto unit vectors.
  \end{enumerate}
\end{remark}

We conclude this section by recalling the following tensor property of the exponential space and of the mixture space, see \cite{MR3365132}.

\begin{proposition} Let us split the components $\reals^n x \mapsto (x_1,x_2) \in \reals^{n_1} \times \reals^{n_2}$ and denote by $M_1$, $M_2$, respectively, the Maxwell densities on the factor spaces.
  \begin{enumerate}
  \item A function $f$ belongs to $\Lexp M$ if and only if for one $\alpha > 0$ the partial integral $x_1 \to \int (\cosh-1)(\alpha f(x_1,x_2)) M(x_2)\ dx_2$ is $M_1$-integrable.  \item A function $f$ belongs to $\LlogL M$ if and only if the partial integral $x_1 \to \int (\cosh-1)_*(f(x_1,x_2)) M(x_2)\ dx_2$ is $M_1$-integrable.  \end{enumerate}
\end{proposition}

\subsection{Gaussian statistical bundle}

It is an essential feature of the exponential affine manifold on $\maxexp M$ discussed in Sec. \ref{sec:affine} that the exponential \emph{statistical bundle}
\begin{equation*}
  S\maxexp M = \setof{(p,U)}{p \in \maxexp M, U \in B_p} \ ,
\end{equation*}
with $B_p = \setof{U \in \Lexp {p\cdot M}}{\expectat {p\cdot M} U = 0}$ is an expression of the tangent bundle in the atlas $\setof{s_p}{p \in \maxexp M}$. This depends on the fact that all fibers $B_p$ are actually a closed subspace of the exponential space $\Lexp M$. This has been proved in Prop. \ref{prop:tangentspaces}. The equality of the spaces $\Lexp {p \cdot M}$ and $\Lexp M$ is equivalent to $p \in \maxexp M$, see the set of equivalent conditions called Portmanteau Theorem in \cite{MR3474821}.

We now investigate whether translation statistical models are sub-set of the maximal exponential model $\maxexp M$ and whether they are sub-manifolds. Proper sub-manifolds of the exponential affine manifold should have a tangent bundle that splits the statistical bundle.

Let $p \in \maxexp M$ and write $f = p \cdot M$. Then $f$ is a positive probability density of the Lebesgue space and so are all its translations
\begin{equation*}
  \tau_h f(x) = p(x-h) M(x-h) = \euler^{h\cdot x - \frac12 \avalof h^2} \tau_h p(x) \cdot M(x) = \tau^*_{-h} p(x) \cdot M(x) \ .
\end{equation*}

From Prop. \ref{prop:SMtoLlogL} and Prop. \ref{prop:taubyh}.2 we know that the translated densities $\tau^*_{-h} p$, are in $\LlogL M$ for all $h \in \reals^n$ and the dependence on $h$ is continuous.

Let us consider now the action of the translation on the values of the chart $s_M$. If $s_M(p) = U$, that is $p = \euler^{U - K_M(U)}$ with $U \in \sdomain M$, then
\begin{multline*}
  \tau^*_{-h} p(x) = \\ \euler^{h\cdot X - \frac12 \avalof h^2} \euler^{U(x-h)-K_M(U)}
  = \expof{h\cdot X - \frac12 \avalof h^2 + \tau_h U - K_M(U)} = \\
  \expof{\left(h\cdot X + \tau_h U - \expectat M {\tau_h U}\right) - \left(K_M(U) + \frac12 \avalof h^2 - \expectat M {\tau_h U}\right)} \ .
\end{multline*}
Here $\tau_h U \in \Lexp M$ because of Prop. \ref{prop:taubyh}.1. If $\tau^*_{-h} p \in \maxexp M$, then
\begin{equation*}
  s_M(\tau^*_{-h} p) = h\cdot X + \tau_h U - \expectat M {\tau_h U} \ .
\end{equation*}
The expected value of the translated $\tau_h U$ is
\begin{equation*}
  \expectat M {\tau_hU} = \gaussint {U(x-h)} x = \euler^{-\frac12 \avalof h^2} \gaussint {\euler^{-h\cdot x}U(x)} x \ .
\end{equation*}

We have found that the action of the translation on the affine coordinate $U = s_M(p)$ of a density $p \in \maxexp M$ is
\begin{equation}\label{eq:Utranslation}
  U \mapsto h\cdot X + \tau_h U - \euler^{-\frac12 \avalof h^2} \expectat M {\euler^{-h\cdot X}U}\ ,
\end{equation}
and we want the resulting value belong to $\sdomain M$, i.e. we want to show that
\begin{multline*}
  \expectat M {\expof{\gamma\left(h\cdot X + \tau_h U - \expectat M {\tau_h U}\right)}} = \\
  \euler^{\gamma\expectat M {\tau_h U}}\expectat M {\euler^{\gamma h \cdot X}} \expectat M {\euler^{\gamma\tau_h U}} = \\
  \euler^{\frac{\gamma^2}2\avalof h^2 + \gamma\expectat M {\tau_h U}} \expectat M {\euler^{\gamma\tau_h U}} \ .
\end{multline*}
is finite for $\gamma$ in a neighborhood of 0.

We have the following result.

\begin{proposition} \begin{enumerate}
\item If $p \in \maxexp M$, for all $h \in \reals^n$ the translated density $\tau^*_{-h}p$ is in $\maxexp M$.
\item 
  If $s_M(p) \in \Cexp M$, then $s_M(\tau^*_{-h}p) \in \Cexp M \cap \sdomain M$ for all $h \in \reals^n$ and dependence in $h$ is continuous.
\end{enumerate}
\end{proposition}

\begin{proof}\begin{enumerate}
\item For each $\gamma$ and conjugate exponents $\alpha,\beta$, we have
  \begin{multline*}
    \expectat M {\euler^{\gamma \tau_h U}} = \euler^{-\frac{1}2\avalof h^2} \gaussint {\euler^{-h \cdot x}\euler^{\gamma U(x)}} x \le \\ \euler^{-\frac{1}2\avalof h^2} \left(\frac1\alpha \euler^{\frac{\alpha^2}2 \avalof h^2} + \frac1\beta \expectat M {\euler^{\beta\gamma U}}\right) \ .
  \end{multline*}
As $U \in \sdomain M$, then $\expectat M {\euler^{\pm aU}} < \infty$ for some $a > 1$, and we can take $\beta = \sqrt a$ and $\gamma = \pm \sqrt a$.
\item 
Under the assumed conditions on $U$ the mapping $h \mapsto \tau_hU$ is continuous in $\Cexp M$ because of Prop. \ref{prop:taubyh}.\ref{item:taubyh2}. So is $h \mapsto \expectat M {\tau_hU}$. As $X_i \in \Cexp M$, $i=1,\dots,n$, the same is true for $h \mapsto h\cdot X$. In conclusion, the translated $U$ of \eqref{eq:Utranslation} belongs to $\Cexp M$.
\end{enumerate}
\qed      
\end{proof}

The proposition above shows that the translation statistical model $\tau^*_{-h} p$, $h \in \mathcal \reals^m$ is well defined as a subset of $\maxexp M$. To check if it is a differentiable sub-manifold, we want to compute the velocity of a curve $t \mapsto \tau^*_{h(t)} p$, that is
\begin{equation*}
  \derivby t \left(h(t)\cdot X + \tau_{h(t)} U - \expectat M {\tau_{h(t)}} U \right) \ .
\end{equation*}
That will require first of all the continuity in $h$, hence $U \in \Cexp M$, and moreover we want to compute $\partial/\partial{h_i} U(x-h)$, that is the gradient of $U$. This task shall be the object of the next section.

Cases other than translations are of interest. Here are two sufficient conditions for a density to be in $\maxexp M$. 

\begin{proposition}\label{prop:sufficientforE}\ 
  \begin{enumerate}
\item Assume $p > 0$ $M$-a.s., $\expectat M p = 1$, and
\begin{equation}\label{eq:boundedmoments}
  \expectat M {p^{n_1/(n_1-1)}} \le 2^{n_1/(n_1-1)}, \quad \expectat M {p^{-1/(n_2-1)}} \le 2^{n_2/(n_2-1)}
  \end{equation}
for some natural $n_1, n_2 > 2$. Then $p \in \maxexp M$, the exponential spaces are equal, $\Lexp M = \Lexp {p\cdot M}$, and for all random variable $U$
\begin{align}
  \normat {\Lexp{p \cdot M}} U &\le 2^{n_1} \normat {\Lexp M} U \ , \label{eq:boundednorms1} \\
  \normat {\Lexp M} U &\le 2^{n_2} \normat {\Lexp {p \cdot M}} {U} \ . \label{eq:boundednorms} 
\end{align}
\item Condition \eqref{eq:boundedmoments} holds for $p = \sqrt{\pi/2}\absoluteval {X_i}$ and for $p = X_i^2$, $i=1,\dots,n$.
\item Let $\chi$ be a diffeomorphism of $\reals^n$ and such that both the derivatives are uniformly bounded in norm. Then the density of the image under $\chi$ of the standard Gaussian measure belongs to $\maxexp M$.
\end{enumerate}  
\end{proposition}

\begin{proof} \begin{enumerate}

\item The bound on the moments in Eq.s \eqref{eq:boundedmoments} is equivalent to the inclusion in $\maxexp M$ because of the definition of $\sdomain M$, or see \cite[Th. 4.7(vi)]{MR3474821}.  Assume $\normat {\Lexp M} U \le 1$, that is $\expectat M {(\cosh-1)(U)} \le 1$. From H\"older inequality and the elementary inequality in Eq. \eqref{eq:simple}, we have
\begin{multline*}
  \expectat {f \cdot M} {(\cosh-1)\left(\frac U{2^{n_1}}\right)} = \expectat M {(\cosh-1)\left(\frac U{2^{n_1}}\right) f} \le \\ \expectat M {(\cosh-1)\left(\frac U{2^{n_1}}\right)^{n_1}}^{1/n_1} \expectat M {f^{n_1/(n_1-1)}}^{(n_1-1)/n_1} \le  \frac 12 \cdot 2 = 1
\end{multline*}
For the other direction, assume $\normat {\Lexp {f \cdot M}} {U} \le 1$, that is $\expectat M {\Phi(U) f} \le 1$, so that

\begin{multline*}
  \expectat M {(\cosh-1)\left(\frac U{2^{n_2}} \right)}
  = \expectat M {(\cosh-1)\left(\frac U{2^{n_2}} \right) f^{1/n_2} f^{-1/n_2}} \le \\ \expectat M {(\cosh-1)\left(\frac U{2^{n_2}} \right)^{n_2} f }^{1/n_2} \expectat M {f^{-1/(n_2-1)}}^{(n_2-1)/n_2} \le \frac 12 \cdot 2 = 1 \ .
\end{multline*}
\item Simple computations of moments.
\item We consider first the case where $\chi(0)=0$, in which case we have the following inequalities. If we define $\alpha^{-1} = \sup \setof{\normof {d\chi(x)}^2}{x \in \reals^n}$, then $\alpha \avalsof{\chi(x)} \le \avalsof x$ for all $x \in \reals^n$ and equivalently, $\alpha \avalsof x \le \avalsof{\chi^{-1}(x)}$. In a similar way, if we define $\beta^{-1} = 
  \sup \setof{\normof{d\chi^{-1}(y)}^2}{y \in \reals^n}$, then $\beta \avalsof{\chi^{-1}(y)} \le \avalsof y$ and $\beta\avalsof{x} \le \avalsof{\chi(x)}$.
  
The density of the image probability is $M\circ\chi^{-1} \absoluteval {\det d\chi^{-1}}$ and we want to show that for some $\epsilon > 0$ the following inequalities both hold, 
  \begin{equation*}
\expectat M {\left(\frac{M\circ\chi^{-1} \absoluteval {\det d\chi^{-1}}}{M}\right)^{1+\epsilon}} < \infty
  \end{equation*}
  and
  \begin{equation*}
    \expectat {M\circ\chi^{-1} \absoluteval {\det d\chi^{-1}}} {\left(\frac{M}{M\circ\chi^{-1} \absoluteval {\det d\chi^{-1}}}\right)^{1+\epsilon}} < \infty \ .
  \end{equation*}

The first condition is satisfied as 
  
\begin{multline*}
  \int \absoluteval {\det d\chi^{-1}(x)}^{1+\epsilon} \left(\frac{M\left(\chi^{-1}(x)\right)}{M(x)}\right)^{1+\epsilon} M(x) \  dx = \\
   \int \absoluteval {\det d\chi^{-1}(x)}^{1+\epsilon} M\left(\chi^{-1}(x)\right)^{1+\epsilon} M(x)^{-\epsilon} \  dx \le \\
(2\pi)^{-n/2} \beta^{-\frac{(1+\epsilon)n}2} \int \expof{-\frac12 \left((1+\epsilon)\absoluteval{\chi^{-1}(x)}^2 - \epsilon\absoluteval x ^2\right)} \  dx  = \\ (2\pi)^{-n/2} \beta^{-\frac{(1+\epsilon)n}2} \int  \expof{-\frac{\avalsof x}2 \left((1+\epsilon)\frac{\avalsof{\chi^{-1}(x)}}{\avalsof x} - \epsilon \right)} \  dx \le \\ (2\pi)^{-n/2} \beta^{-\frac{(1+\epsilon)n}2} \int \expof{-\frac{\avalsof x}2 \left((1+\epsilon)\alpha - \epsilon \right)} \  dx  \ ,
\end{multline*}
where we have used the Hadamard's determinant inequality
\begin{equation*}
  \avalof{\det d\chi^{-1}(x)} \le \normof{d\chi^{-1}(x)}^n \le \beta^{-n/2}  
\end{equation*}
and the lower bound $\alpha \le \frac{\avalsof{\chi^{-1}(x)}}{\avalsof x}$, $x \in \reals^n_*$. If $\alpha \ge 1$ then $(1+\epsilon)\alpha - \epsilon = \alpha + \epsilon(\alpha-1) \ge \alpha > 0$ for all $\epsilon$. If $\alpha < 1$, then $(1+\epsilon)\alpha-\epsilon > 0$ if $\epsilon < \alpha/(1-\alpha)$ e.g., $\epsilon=\alpha/2(1-\alpha)$, which in turn gives $(1+\epsilon)\alpha - \epsilon = \alpha/2$. In conclusion, there exist an $\epsilon > 0$ such that
\begin{multline*}
  \int \absoluteval {\det d\chi^{-1}(x)}^{1+\epsilon} \left(\frac{M\left(\chi^{-1}(x)\right)}{M(x)}\right)^{1+\epsilon} M(x) \  dx \le \\
(2\pi)^{-n/2} \avalof{\det d\chi^{-1}(x)}^{1+\epsilon} \int  \expof{-\frac{\alpha\avalsof x}4} \  dx  = \left(\frac\alpha2\right)^{n/2} \ .
\end{multline*}

For the second inequality, 

\begin{multline*}   \int \left(\frac{M(y)}{M\left(\chi^{-1}(y)\right) \avalof {\det d\chi^{-1}(y)}}\right)^{1+\epsilon} M\left(\chi^{-1}(y)\right) \avalof {\det d\chi^{-1}(y)} \ dy = \\
  \int M(y)^{1+\epsilon} M\left(\chi^{-1}(y)\right)^{-\epsilon} \avalof {\det d\chi^{-1}(y)}^{-\epsilon} \ dy = \\
  \int M(\chi(x))^{1+\epsilon} M(x)^{-\epsilon} \avalof {\det d\chi^{-1}(\chi(x))}^{-\epsilon} \avalof{\det d\chi(x)} \ dx = \\
 \int \avalof{\det d\chi(x)}^{1+\epsilon} M(\chi(x))^{1+\epsilon} M(x)^{-\epsilon} \ dx \ .
\end{multline*}

As the last term is equal to the expression in the previous case with $\chi^{-1}$ replaced by $\chi$, the same proof applies with the bounds $\alpha$ and $\beta$ exchanged.
\end{enumerate}
\qed
\end{proof}

\begin{remark}
  While the moment condition for proving $p \in \maxexp M$ has been repeatedly used, nonetheless the results above have some interest. The first one is an example where an explicit bound for the different norms on the fibers of the statistical bundle is derived. The second case is the starting point for the study of transformation models where a group of transformation $\chi_\theta$ is given.   
\end{remark}

\section{Weighted Orlicz--Sobolev Model Space}
\label{sec:orlicz-sobolev}

We proceed in this section to the extension of our discussion of translation statistical models to statistical models of the Gaussian space whose densities are differentiable. We restrict to generalities and refer to previous work in \cite{MR3365132} for examples of applications, such as the discussion of Hyv\"arinen divergence. This is a special type of divergence between densities that involves an $L^2$-distance between gradients of densities \cite{MR2249836} which has multiple applications. In particular, it is related with the improperly called Fisher information in \cite[p. 49]{MR2409050}.

We are led to consider a case classical weighted Orlicz--Sobolev spaces which is not treated in much detail in standard monographs such as \cite{MR2424078}. The analysis of the finite dimensional Gaussian space i.e. the space of square-integrable random variables on $\left(\reals^n,\mathcal B(\reals^n),M\right)$ is a well developed subject. Some of the result below could be read as special case of that theory. We refer to P. Malliavin's textbook \cite[Ch 5]{MR1335234} and to D. Nualart's monograph \cite{MR2200233}.

\subsection{Orlicz-Sobolev spaces with Gaussian weight}
\label{sec:orlicz-sobol-spac}

The first definitions are taken from our \cite{MR3365132}. 

\begin{definition}
The exponential and the mixture Orlicz-Sobolev-Gauss (OSG) spaces are, respectively,
\begin{align}
   \Wexp M &= \setof{f \in \Lexp M}{\partial_j f \in \Lexp M}   \label{eq:OrSobExp} \ ,\\
  \WlogL M &= \setof{f \in \LlogL M}{\partial_j f \in \LlogL M}   \label{eq:OrSobLog} \ ,
\end{align}
where $\partial_j$, $ j = 1, \dots, n$, is the partial derivative in the sense of distributions. 
\end{definition}

As $\phi \in \Cinfcs{\reals^n}$ implies $\phi M \in \Cinfcs{\reals^n}$, for each $f \in \WlogL M$ we have, in the sense of distributions, that 
\begin{equation*}
\scalarat M {\partial_jf} \phi = \scalarof{\partial_j f}{\phi M} = - \scalarof{f}{\partial_j(\phi M)} = \scalarof {f}{M(X_j - \partial_j)\phi} = \scalarat M f {\delta_j \phi} \ ,
\end{equation*}
with $\delta_j \phi = (X_j - \partial_j)\phi$. The \emph{Stein operator} $\delta_i$ acts on $\Cinfcs{\reals^n}$.

The meaning of both operators $\partial_j$ and $\delta_j = (X_j - \partial_j)$ when acting on square-integrable random variables of the Gaussian space is well known, but here we are interested in the action on OSG-spaces. Let us denote by $\Cinfp {\reals^n}$ the space of infinitely differentiable functions with polynomial growth. Polynomial growth implies the existence of all $M$-moments of all derivatives, hence $\Cinfp{\reals^n} \subset \WlogL M$. If $f \in \Cinfp{\reals^n}$, then the distributional derivative and the ordinary derivative are equal and moreover $\delta_jf \in \Cinfp{\reals^n}$. For each $\phi \in \Cinfcs{\reals^n}$ we have $\scalarat M{\phi}{\delta_j f} = \scalarat M {\partial_j \phi} f$.

The OSG spaces $W_{\cosh-1}^1(M)$ and $W_{(\cosh-1)_*}^1(M)$ are both Banach spaces, see \cite[Sec. 10]{MR724434}. In fact, both the product functions $(u,x) \mapsto (\cosh-1)(u)M(x)$ and $(u,x) \mapsto (\cosh-1)_*(u)M(x)$ are $\phi$-functions according the Musielak's definition. The norm on the OSG-spaces are the graph norms,
\begin{equation}
  \label{eq:OrSob-norm}
  \normat {W^{1}_{(\cosh-1)}(M)} f = \normat {\Lexp M} f + \sum_{j=1}^n \normat {\Lexp M} {\partial_j f}
\end{equation}
and
\begin{equation}
  \label{eq:preOrSob-norm}
  \normat {W^{1}_{(\cosh-1)_*}(M)} f = \normat {\Lexp M} f + \sum_{j=1}^n \normat {\Lexp M} {\partial_j f} \ .
\end{equation}

Because of Prop. \ref{prop:tangentspaces}, see also \cite[Th. 4.7]{MR3474821}, for each $p \in \maxexp M$, we have both equalities and isomorphisms $\Lexp {p\cdot M} = \Lexp M$ and $\LlogL (p\cdot M) = \LlogL M$. It follows
\begin{align}
   \Wexp M &= \Wexp {p\cdot M} \notag \\ &= \setof{f \in \Lexp p}{\partial_j f \in \Lexp {p\cdot M}} \ , \label{eq:OrSobexpatp} \\
  \WlogL M &= \WlogL {p\cdot M} \notag \\ &= \setof{f \in \LlogL {p\cdot M}}{\partial_j f \in \LlogL p}   \label{eq:OrSobLogatp} \ ,
\end{align}
and equivalent graph norms for any density $p \in \maxexp M$. The OSG spaces are compatible with the structure of the maximal exponential family $\maxexp M$. In particular, as all Gaussian densities of a given dimension belong into the same exponential manifold, one could have defined the OSG spaces with respect to any of such densities.

We review some relations between OSG-spaces and ordinary Sobolev spaces. For all $R > 0$
\begin{equation*}
(2\pi)^{- \frac n2} \ge  M(x) \ge M(x) (\absoluteval x < R) \ge (2\pi)^{- \frac n2} \euler^{-\frac {R^2}2} (\absoluteval x < R), \quad x \in \reals^n.
\end{equation*}
\begin{proposition} \label{eq:embeddings}
  Let $R > 0$ and let $\Omega_R$ denote the open sphere of radius $R$.
  \begin{enumerate}
  \item We have the continuous mappings
    \begin{equation*}
      \Wexp{\reals^n} \subset \Wexp M \rightarrow W^{1,p}(\Omega_R), \quad p \ge 1.
     \end{equation*}
    \item We have the continuous mappings
      \begin{equation*}
        W^{1,p}(\reals^n) \subset \WlogL{\reals^n} \subset \WlogL M \rightarrow W^{1,1}(\Omega_R), \quad p > 1.  
      \end{equation*}
\item \label{item:embeddings3}Each $u \in \Wexp M$ is a.s. H\"older of all orders on each $\overline\Omega_R$ and hence a.s. continuous. The restriction $\Wexp M \to C(\overline\Omega_R)$ is compact. 
 \end{enumerate}
\end{proposition}
\begin{proof}
  \begin{enumerate}
  \item From the inequality on $M$ and from $(\cosh-1)(y) \ge y^{2n}/(2n)!$.
  \item From the inequality on $M$ and from $y^2/2 \ge (\cosh-1)_*(y)$ and $\cosh(1) - 1 + (\cosh-1)_*(y) \ge \absoluteval y$.
  \item It is the Sobolev's embedding theorem \cite[Ch. 9]{MR2759829}.
  \end{enumerate}
\end{proof}

Let us consider now the extension of the $\partial_j$ operator to the OSG-spaces and its relation with the translation operator.

The operator given by the ordinary partial derivative $\partial_j \colon \Cinfp{\reals^n} \to \Cinfp{\reals^n} \subset \LlogL M$ is closable. In fact, if both $f_n \to 0$ and $\partial_j f_n \to \eta$ in $\LlogL M$, then for all $\phi \in \Cinfcs{\reals^n}$,
\begin{equation*}
 \scalarat M \phi \eta =  \lim_{n\to\infty} \scalarat M {\phi} {\partial_j f_n} = \lim_{n\to\infty} \scalarat M {\delta \phi} {f_n} = 0 \ , 
\end{equation*}
hence $\eta=0$. The same argument shows that $\partial_j \colon \Cinfcs{\reals^n} \to \Cinfcs{\reals^n}  \subset \Lexp M$ is closable.

For $f \in \Lexp M$ we define $\tau_h f(x) = f(x-h)$ and it holds $\tau_hf \in \Lexp M$ because of Prop. \ref{prop:taubyh}\eqref{item:taubyh1}. For each given $f \in \Wexp M$ we denote by $\partial_j f \in \Wexp M$, $j=1,\dots,n$ its distributional partial derivatives and write $\nabla f =(\partial_j f \colon j=1,\dots,n)$.

\begin{proposition}[Continuity and directional derivative] \ 
  \begin{enumerate}
  \item For each $f \in \Wexp M$, each unit vector $h \in S^n$, and all $t \in \reals$, it holds
    \begin{equation*}
      f(x+th) - f(x) = t \int_0^1 \sum_{j=1}^n \partial_j f(x + sth)h_j \ ds \ . 
    \end{equation*}
Moreover, $\absoluteval t \le \sqrt 2$ implies 
\begin{equation*}
  \normat {\Lexp M} {f(x+th) - f(x)} \le 2 t \normat {\Lexp M} {\nabla f} \ ,
\end{equation*}
especially, $\lim_{t\to0} \normat {\Lexp M} {f(x+th) - f(x)} = 0$ uniformly in $h$.
  \item For each $f \in \Wexp M$ the mapping $h \mapsto \tau_h f$ is differentiable from $\reals^n$ to $L^{\infty-0}(M)$ with gradient $\nabla f$ at $h=0$.
  \item For each $f \in \Wexp M$ and each $g \in \LlogL M$, the mapping $h \mapsto \scalarat M {\tau_h f} g$ is differentiable. Conversely, if $f \in \Lexp M$ and $h \mapsto \tau_h f$ is weakly differentiable, then $f \in \Wexp M$ 
  \item If $\partial_j f \in \Cexp M$, $j = 1,\dots,n$, then strong differentiability in $\Lexp M$ holds. 
  \end{enumerate}
\end{proposition}

\begin{proof} \ 
  \begin{enumerate}
  \item Recall that for each $g \in C_0^\infty(\reals^n)$ we have
    \begin{equation*}
      \scalarat M {\partial_j f} g = - \scalarof f {\partial_j(gM)} = \scalarat M f {\delta_j g}, \quad \delta_j g = X_jg - \partial_j g \in \Cinfcs{\reals^n} \ .
    \end{equation*}
We show the equality $\tau_{-th}f - f = t \int_0^1 \tau_{-sth}(\nabla f) \cdot h \ ds$ in the scalar product with a generic $g \in C_0^\infty(\reals^n)$:
\begin{align*}
  \scalarat M {\tau_{-th}f - f} g &= \int f(x+th)g(x)M(x)\ dx - \int f(x)g(x)M(x)\ dx \\ &= \int f(x)g(x-th)M(x-th)\ dx - \int f(x)g(x)M(x)\ dx \\
  &= \int f(x)\left(g(x-th)M(x-th)-g(x)M(x)\right)\ dx \\
&= -t \int f(x) \int_0^1 \sum_{j=1}^n \partial_j (gM)(x - sth)h_j \ ds \ dx \\
&= -t \int_0^1 \int f(x) \sum_{j=1}^n \partial_j (gM)(x - sth)h_j \ dx \ ds \\
&= t \int_0^1 \int \sum_{j=1}^n \partial_j f(x) h_j \ g(x - sth) M(x - sth) \ dx \ ds \\
&= t \int_0^1 \int \sum_{j=1}^n \partial_j f(x + sth)h_j \ g(x) M(x) \ dx \ ds \\
&= \scalarat M {t \int_0^1 \tau_{-sth}(\nabla f) \cdot h \ ds} g \ .
\end{align*}

If $\absoluteval t \le \sqrt {\log 2}$ then the translation $sth$ is small, $\absoluteval{sth} \le \sqrt {\log 2}$ so that, according to Prop. \ref{prop:taubyh}\eqref{item:taubyh1}, we have$\normat {\Lexp M} {\tau_{-sth}(\nabla f \cdot h)} \le 2 \normat {\Lexp M} {\nabla f \cdot h}$ and the thesis follows.

\item We want to show that the following limit holds in all $L^\alpha(M)$-norms, $\alpha > 1$:
  \begin{equation*}
    \lim_{t\to0} \frac{\tau_{-th}f - f}{t} = \sum_{j=1}^n h_j \partial_j f\ .
  \end{equation*}
Because of the identity in the previous Item, we need to show the limit
\begin{equation*}
  \lim_{t\to0} \int \absoluteval{\int_0^1 (\tau_{-sth}(\nabla f(x) \cdot h) - \nabla f(x) \cdot h)\ ds}^\alpha \ M(x)dx = 0 \ .
\end{equation*}
The Jensen's inequality gives 
\begin{multline*}
  \int \absoluteval{\int_0^1 (\tau_{-sth}(\nabla f(x) \cdot h) - \nabla f(x) \cdot h)\ ds}^\alpha \ M(x)dx \le \\ \int_0^1 \int \absoluteval{\tau_{-sth}(\nabla f(x) \cdot h) - \nabla f(x) \cdot h)}^\alpha \ M(x)dx \ ds 
\end{multline*}
and the result follows because translations are bounded and the continuous in $L^\alpha(M)$.

\item We have
  \begin{multline*}
    \scalarat M {\int_0^1 (\tau_{(-sth)}f - f) \ ds} g = \int_0^1 \scalarat M {\tau_{(-sth)}f - f} g \ ds = \\
\int_0^1 \scalarat M f {\tau^*_{(-sth)}g - g} \ ds \ .
  \end{multline*}
Conclusion follows because $y \mapsto \tau_y^* g$ is bounded continuous.

Assume now $f \in \Lexp M$ and $h \mapsto \tau_h f$ is weakly differentiable. Then there exists $f_1,\dots,f_n \in \Lexp M$ such that for each $\phi \in C^\infty_0(\reals^n)$
\begin{multline*}
  \scalarof {f_j} {\phi M} = \scalarat M {f_j} \phi = \derivby t \scalarat M {\tau_{-te_j}f} \phi = \derivby t \scalarof {\tau_{-te_j}f} {\phi M} = \\ \derivby t \scalarof {f} {\tau_{te_j}(\phi M)} = - \scalarof f {\partial_j (\phi M} \ .
\end{multline*}
The distributional derivative holds because $\phi M$  is the generic element of $C^\infty_0(\reals^n)$.

\item For each $\rho > 0$ Jensen's inequality implies
\begin{multline*}
  \normat {\Lexp M} {\int_0^1 (\tau_{-sth}(\nabla f \cdot h) - \nabla f \cdot h)\ ds \ M(x)dx} \le \\  \int_0^1 \normat {\Lexp M} {(\tau_{-sth}(\nabla f \cdot h) - \nabla f \cdot h)\ M(x)dx} \ ds \ .
\end{multline*}
As in Prop. \ref{prop:taubyh}\eqref{item:taubyh1} we choose $\absoluteval t \le \sqrt {\log 2}$ to get $\absoluteval {ste_j} \le \sqrt {\log 2}$, $0 \le s \le 1$, so that $\normat {\Lexp M}{\tau_{-sth}\nabla f \cdot h} \le 2 \normat {\Lexp M}{\nabla f \cdot h}$, hence the integrand is bounded by $\normat {\Lexp M}{\nabla f \cdot h}$. The convergence for each $s$ follows from the continuity of the translation on $\Cexp M$.
  \end{enumerate}
\end{proof}

Notice that in Item 2. of the proposition we could have derived a stronger differentiability if the mapping $h \mapsto \tau_h \nabla f$ were continuous in $\Lexp M$. That, and other similar observations, lead to the following definition.

\begin{definition}
  The \emph{Orlicz-Sobolev-Gauss exponential class} is 
  \begin{equation*}
    \WCexp M = \setof{f \in \Wexp M}{f,\partial_j f \in \Cexp M, j=1,\dots,n}
  \end{equation*}
\end{definition}

The following density results will be used frequently in approximation arguments. We denote by $(\omega_n)_{n\in\naturals}$ a sequence of mollifiers.
\begin{proposition}[Calculus in $\WCexp M$]\label{prop:molliD}
\begin{enumerate}
\item \label{item:molliD1} For each $f \in \WCexp M$ the sequence $f*\omega_n$, $n \in \naturals$, belongs to $C^\infty(\reals^n) \cap \Wexp M$. Precisely, for each $n$ and $j=1,\dots,n$, we have the equality $\partial_j (f * \omega_n) = (\partial_j f) * \omega_n$; the sequences $f * \omega_n$, respectively $\partial_j f * \omega_n$, $j=1,\dots,n$, converge to $f$, respectively $\partial_j f$, $j=1,\dots,n$, strongly in $\Lexp M$. 
\item \label{item:molliD2}
Same statement is true if $f \in \WlogL M$.
\item \label{item:molliD3}
Let be given $f \in \WCexp M$ and $g \in \WlogL M$. Then $fg \in W^{1,1}(M)$ and $\partial_j (fg) = \partial_j f g + f \partial_j g$. 
\item \label{item:molliD4} Let be given $F \in C^1(\reals)$ with $\normat \infty {F'} < \infty$. For each $U \in \WCexp M$, we have $F\circ U, F'\circ U \partial_j U \in \Cexp M$ and  $\partial_j F\circ U = F'\circ U \partial_j U$, in particular $F(U) \in \WCexp M$.
\end{enumerate}
\end{proposition}

\begin{proof}
  \begin{enumerate}
  \item We need only to note that the equality $\partial_j (f * \omega_n) = (\partial_j f) * \omega_n$ is true for $f \in \Wexp M$. Indeed, the sequence $f*\omega_n$ belongs to $C^\infty(\reals^n) \cap \Lexp M$ and converges to $f$ in $\Lexp M$-norm according from Prop. \ref{prop:molli}. The sequence $\partial_j f*\omega_n = (\partial_j f)*\omega_n$ converges to $\partial_j f$ in $\Lexp M$-norm because of the same theorem.
  \item Same proof.
  \item Note that $fg, \partial_j f g + f \partial_j g \in L^1(M)$. The following converge in $L^1(M)$ holds
    \begin{equation*}
    \partial_j f g + f \partial_j g = \lim_{n\to\infty} \partial_j f*\omega_n g*\omega_n + f*\omega_n \partial_j *\omega_n = \lim_{n\to\infty} \partial_j f*\omega_n g*\omega_n \ , 
    \end{equation*}
so that  for all $\phi \in C^\infty_0(\reals^n)$
\begin{multline*}
\scalarof {\partial_j f g + f \partial_j g} \phi = \lim_{n\to\infty} \scalarof {\partial_j f*\omega_n g*\omega_n} \phi = \\ \lim_{n\to\infty} - \scalarof {f*\omega_ng*\omega_n}{\partial_j \phi} = - \scalarof {fg}{\partial_j \phi} \ .  
\end{multline*}
It follows that the distributional partial derivative of the product is $\partial_jfg = \partial_j f g + f \partial_j g$, in particular belongs to $L^1(M)$, hence $fg \in W^{1,1}(M)$.
\item From the assumption on $F$ we have $\absoluteval {F(U)} \le \absoluteval{F(0)} + \normat \infty {F'} \absoluteval U$. It follows $F\circ U \in \Lexp M$ because
    \begin{multline*}
      \int (\cosh-1)\left(\rho F(U(x))\right) \ M(x)dx \le \\ \frac12 (\cosh-1)(2\rho F(0)) + \frac12 \int (\cosh-1)\left(2 \rho \normat \infty {F'} U(x))\right) \ M(x)dx \ ,
    \end{multline*}
and $\rho \normat {\Lexp M} {F(U)} \le 1$ if both
\begin{equation*}
  (\cosh-1)(2\rho F(0)) \le 1, \quad  2 \rho \normat \infty {F'} \normat {\Lexp M} U \le 1 \ .
\end{equation*}
In the same way we show that $F'\circ U \partial_j U \in \Lexp M$. Indeed,
\begin{multline*}
\int (\cosh-1)\left(\rho F'(U(x)) \partial_j U(x) \right) \ M(x)dx \le \int (\cosh-1)\left(\rho \normat \infty {F'} \partial_j U(x) \right) \ M(x)dx \ ,
\end{multline*}
so that $\rho \normat {\Lexp M} {F'\circ U \partial_j U} \le 1$ if $\rho \normat \infty {F'} \normat {\Lexp M} {\partial_j U(x)} = 1$.
Because of the Item \eqref{item:molliD1} the sequence $U*\omega_n$ belongs to $C^\infty$ and converges strongly in $\Lexp M$ to $U$, so that from
\begin{equation*}
  \normat {\Lexp M} {F\circ(U*\omega_n) - F\circ U} \le \normat \infty {F'} \normat {\Lexp M} {U*\omega_n-U}
\end{equation*}
we see that $F\circ (U*\omega_n) \to F\circ U$ in $\Lexp M$. In the same way,
\begin{multline*}
\normat {\Lexp M} {F'\circ(U*\omega_n) \partial_j (U*\omega_n) - F'\circ U \partial_j U} \le \\ \normat {\Lexp M}{F'\circ(U*\omega_n) (\partial_j (U*\omega_n) - \partial_j U)} \\ + \normat {\Lexp M} {(F'\circ(U\circ\omega_n) - F'\circ U) \partial_j U} \le \\
\normat \infty {F'} \normat {\Lexp M}{\partial_j (U*\omega_n) - \partial_j U} + \\ \normat {\Lexp M} {(F'\circ(U\circ\omega_n) - F'\circ U) \partial_j U} \ .
\end{multline*}
The first term goes clearly to 0, while the second term requires consideration. Note the bound
\begin{equation*}
\absoluteval {(F'\circ(U\circ\omega_n) - F'\circ U) \partial_j U} \le 2 \normat \infty {F'} \absoluteval {\partial_j U} \ ,
\end{equation*}
so that the sequence $(F'\circ(U\circ\omega_n) - F'\circ U) \partial_j U$ goes to zero in probability and is bounded by a function in $\Cexp M$. This in turn implies the convergence in $\Lexp M$.

Finally we check that  the distributional derivative of $F\circ U$ is $F'\circ U \partial_j U$: for each $\phi \in C^\infty_0(\reals^n)$
\begin{align*}
  \scalarof {\partial_j F\circ U} {\phi M} &=  - \scalarof {F\circ U} {\partial_j (\phi M)} \\
 &= - \scalarat M {F\circ U} {\delta_j \phi} \\
&= \lim_{n\to\infty} \scalarat M  {F\circ (U*\omega_n)} {\delta_j \phi} \\
&= \lim_{n\to\infty} \scalarat M  {\partial_j F\circ (U*\omega_n)} {\phi} \\
&= \lim_{n\to\infty} \scalarat M  {F'\circ (U*\omega_n) \partial_j (U*\omega_n)} {\phi} \\ &= \scalarat M  {F'\circ U \partial_j U} {\phi} \\
&= \scalarof {F'\circ U \partial_j U} {\phi M} \ .
 \end{align*}
\end{enumerate}
\end{proof}

We conclude our presentation by re-stating a technical result from \cite[Prop. 15]{MR3365132}, where the assumptions where not sufficient for the stated result.

\begin{proposition}\label{prop:feuler}\ 
  \begin{enumerate}
  \item  If $U \in \sdomainat M$ and $f_1,\dots,f_m \in \Lexp M$, then $f_1 \cdots f_m \euler^{M-K_M(M)} \in L^\gamma(M)$  for some $\gamma > 1$, hence it is in $\LlogL M$.
  \item  If $U \in \sdomainat M \cap \WCexp M$ and $f \in \WCexp M$, then
    \begin{equation*}
      f\euler^{u-K_M(u)} \in \WlogL M \cap C(\reals^n) \ ,
    \end{equation*}
and its distributional partial derivatives are $(\partial_j f + f \partial_j u)\euler^{u - K_M(u)}$ 
  \end{enumerate}
\end{proposition}

\begin{proof} 
\begin{enumerate}
\item
From We know that $\euler^{U-K_M(U)} \cdot M \in \maxexp M$ and $\euler^{U-K_M(U)} \in L^{1+\varepsilon}(M)$ for some $\varepsilon > 0$. From that, let us prove that $f_1 \cdots f_m \euler^{U-K_M(U)} \in L^\gamma(M)$ for some $\gamma > 1$. According to classical (m+1)-term Fenchel-Young inequality,
  \begin{multline*}
    \absoluteval{f_1(x) \cdots f_n(x)} \euler^{U(x)-K_M(U)} \leq \\ 
\sum_{i=1}^m \frac{1}{\alpha_i}\absoluteval{f_i(x)}^{\alpha_i} + \frac{1}{\beta}\absoluteval{\euler^{U(x)-K_M(U)}}^{\beta}, \\ \alpha_1,\dots,\alpha_m,\beta > 1, \sum_{i=1}^m \frac{1}{\alpha_i}+\frac{1}{\beta}=1, x \in \reals^n.
  \end{multline*}
Since $(\cosh-1)_{*}$ is convex, we have
\begin{multline*}
  \expectat M{(\cosh-1)_{*}(\absoluteval{f_1 \cdots f_m}\euler^{U-K_M(U)})} \leq \\ \sum_{i=1}^m \frac{1}{\alpha_i}\expectat M {(\cosh-1)_{*}(\absoluteval{f_i}^{\alpha_i})} + \frac{1}{\beta} \expectat M{(\cosh-1)_{*}\left(\euler^{\beta(U-K_M(U))}\right)}.
\end{multline*}
Since $f_1, \dots, f_m \in \Lexp M \subset \cap_{\alpha > 1} L^\alpha(M)$, one has $|f_i|^{\alpha_i} \in \LlogL M$, for $i=1,\dots,m$ and all $\alpha_i > 1$, hence $\expectat M{(\cosh-1)_{*}(\absoluteval{f_i}^{\alpha_i})}  < \infty$ for $i=1,\dots,m$ and all $\alpha_i>1$. By choosing $1<\beta< 1+\varepsilon$ one has $\euler^{\beta(U(x)-K_M(U))} \in L^{\gamma}(M) \subset \LlogL M$, $\gamma = \frac{1+\varepsilon}{\beta}$, so that $\expectat M{(\cosh-1)_{*}\left(\euler^{\beta(U-K_M(U))}\right)} < \infty$. This proves that $(\cosh-1)_{*}(f_1 \cdots f_m\euler^{U - K_M(U)}) \in L^{1}(M)$, which implies $f_1 \cdots f_m \euler^{U(x)-K_M(U)} \in \LlogL M$.
\item
From the previous item we know $f \euler^{U-K_M(U)} \in \LlogL M$. For each $j=1,\dots,n$ from prop. \ref{prop:molliD}\eqref{item:molliD3} we have the distributional derivative $\partial_j (f \euler^{U}) = \partial f \euler^{U} + f \partial_j \euler^{U-K_M(U)}$ we we need to show a composite function derivation, namely $\partial_j \euler^{U-K_M(U)} = \partial_j u \euler^{U-K_M(U)}$. Let $\chi \in C_0^\infty(\real^n)$ be a cut-off equal to 1 on the ball of radius 1, zero outside the ball of radius 2, derivative bounded by 2, and for $n \in \naturals$ consider the function $x \mapsto F_n(x) = \chi(x/n)\euler^{x}$ which is $C^\infty(\reals^n)$ and whose derivative is bounded:
\begin{equation*}
  F_n'(x) = \left(\frac1n \chi'(x/n)+\chi(x/n)\right) \euler^x \le \left(\frac2n + 1\right)\euler^{2n} \ .
\end{equation*}
As Prop. \ref{prop:molliD}\eqref{item:molliD4} applies, we have $\partial_j F_n(U) = F'_n(U) \partial_j U \in \Cexp M$. Finally, for each $\phi \in C^\infty_0(\reals^n)$, 
\begin{align*}
  \scalarof {\partial_j \euler^U} \phi &= - \scalarof {\euler^U}{\partial_j \phi} \\ &= - \lim_{n\to\infty} \scalarof {F_n(U)} {\partial_j \phi} \\
&= \lim_{n\to\infty} \scalarof {\partial F_n(U)} \phi \\
&= \lim_{n\to\infty} \scalarof {(\frac1n \chi'(U/n)+\chi(U/n))\partial_j U \euler^U} \phi \\ &= \scalarof {\partial_j U \euler^U} \phi \ .
\end{align*}
\end{enumerate}
\end{proof}

\begin{remark}
\label{rem:wheref}
 As a particular case of the above proposition, we see that $U \in \sdomainat M \cap \WCexp M$ implies
 \begin{equation*}
   \euler^{U-K_{M}(U)} \in \WlogL M \qquad \text{ with } \quad \bnabla \euler^{U-K_M(U)} = \bnabla u\, \euler^{U-K_M(U)} \ .
 \end{equation*}
\end{remark}

\section{Conclusions}
\label{sec:conclusion}

In this paper we have given a self-contained expositions of the Exponential Affine Manifold on the Gaussian space. The Gaussian assumption allows to discuss topics that are not available in the general case, where the geometry of the sample space has no role.

In particular, we have focused on the action of translations on the probability densities of the manifold and on properties of their derivatives. Other related results, such as Poincar\'e-type inequalities, have been discussed.

Intended applications are those already discussed in \cite{MR3365132}, in particular Hyv\"arinen divergence and other statistical divergences involving derivatives, together with their gradient flows.

\subsubsection*{Acknowledgments}

The author acknowledges support from the \emph{de Castro Statistics Initiative}, Collegio Carlo Alberto, Moncalieri, Italy. He is a member of INdAM/GNAMPA.

\bibliographystyle{splncs03.bst}

\end{document}